\documentclass[a4paper]{article}
\usepackage[english]{babel}
\usepackage{amsmath,amssymb,amsthm,exscale,amsopn,mathtools,fullpage}

\newtheorem{theorem}{Theorem}[section]
\newtheorem{lemma}[theorem]{Lemma}
\newtheorem{proposition}[theorem]{Proposition}
\newtheorem{corollary}[theorem]{Corollary}

\newtheorem{remark}[theorem]{Remark}

\newtheorem{defi}[theorem]{Definition}
\newtheorem{stw}[theorem]{Proposition}
\newtheorem{question}[theorem]{Question}

\newcommand{\supp}{\operatorname{supp}}

\newcommand{\pref}{\operatorname{pref}}
\newcommand{\suff}{\operatorname{suff}}

\newcommand{\s}{\subseteq}
\newcommand{\arr}{\xrightarrow}

\DeclareMathOperator{\HK}{HK}

\begin{document}

\begin{titlepage}

\title{Combinatorics and structure of Hecke--Kiselman algebras}

\author{J. Okni\'nski, M. Wiertel}
\date{}
\maketitle

\begin{abstract}
Hecke--Kiselman monoids $\HK_{\Theta}$ and their algebras
$K[\HK_{\Theta}]$, over a field $K$, associated to finite oriented
graphs $\Theta$ are studied. In the case $\Theta $ is a cycle of
length $n\geqslant 3$, a hierarchy of certain unexpected
structures of matrix type is discovered within the monoid
$C_n=\HK_{\Theta}$ and it is used to describe the structure and
the properties of the algebra $K[C_n]$. In particular, it is shown
that $K[C_n]$ is a right and left Noetherian algebra, while it has
been known that it is a PI-algebra of Gelfand--Kirillov dimension
one. This is used to characterize all Noetherian algebras
$K[\HK_{\Theta}]$ in terms of the graphs $\Theta$. The strategy of
our approach is based on the crucial role played by submonoids of
the form $C_n$ in combinatorics and structure of arbitrary
Hecke--Kiselman monoids $\HK_{\Theta}$.
\end{abstract}

\vspace{20pt}



\vspace{50pt}

\begin{tabular}{lll}
Jan Okni\'nski & \quad \quad \quad \quad \quad & Magdalena Wiertel \\
okninski@mimuw.edu.pl & & M.Wiertel@mimuw.edu.pl \\
 & & \\
  & & \\
Institute of Mathematics & & \\
University of Warsaw & & \\
Banacha 2 & & \\

02-097 Warsaw, Poland & &
\end{tabular}

\vspace{100pt}

 \noindent {\bf Acknowledgment.} This work was
supported by  grant 2016/23/B/ST1/01045 of the National Science
Centre (Poland).
\end{titlepage}

\section{Introduction} \label{intro}

For an arbitrary finite simple digraph $\Theta$ with $n$ vertices
$\{1, \ldots, n\}$, a finitely generated monoid $\HK_{\Theta}$ was
defined by Ganyushkin and Mazorchuk in \cite{maz} by specifying
generators and relations. Namely,
\begin{itemize}
    \item[(i)] $\HK_{\Theta}$ is generated by elements $ x_i=x_{i}^2$,
    where $1 \leqslant i \leqslant n$,
    \item[(ii)] if the vertices $i$, $j$ are not connected in  $\Theta$,
    then  $x_ix_j = x_jx_i$,
    \item[(iii)] if $i$, $j$ are connected by an arrow $i \to j$ in $\Theta$,
    then $x_ix_jx_i = x_jx_ix_j = x_ix_j$,
        \item[(iv)] if $i$, $j$ are connected by an (unoriented) edge in $\Theta$,
        then $x_ix_jx_i = x_jx_ix_j$.
\end{itemize}

If the graph $\Theta$ is unoriented (has no arrows), the monoid
$\HK_{\Theta}$ is isomorphic to the so-called $0$-Hecke monoid
$H_0(W)$, where $W$ is the Coxeter group of the graph $\Theta$.
Because of its strong connection to the Coxeter group, \cite{tsa},
and to the corresponding Hecke algebra, \cite{mathas}, the latter
monoid plays an important role in representation theory. One of
the reasons for the interest in the Hecke--Kiselman monoids is
that they are natural quotients of the Hecke monoids. It is worth
mentioning that relations of the above types arise also in a
natural way in certain other contexts of representation theory,
\cite{grensing}. If $K$ is a field, then by the Hecke--Kiselman
algebra defined by $\Theta$ we mean the semigroup algebra
$K[\HK_{\Theta}]$. In other words, this is the $K$--algebra
defined by the above presentation. Several combinatorial
properties of the Hecke--Kiselman monoids, and their
representations, have been studied in \cite{type
A},\cite{maz},\cite{OnK}. The aim of this paper is to continue the
study of the algebra $K[\HK_{\Theta}]$ in the case when $\Theta$
is an oriented graph, started in~\cite{mecel_okninski1}. In
particular, a version of growth alternative has been obtained
there and algebras of finite Gelfand--Kirillov dimension have been
characterized. If $\Theta $ is oriented, it is also well known
that $\HK_{\Theta}$ is finite if and only if the graph $\Theta$ is
acyclic. Because of these results, it seems that the structure of
the algebras $K[C_n]$, where $C_n$ is the Hecke--Kiselman monoid
corresponding to the oriented cycle of length $n$, is crucial for
understanding the structure and properties of arbitrary algebras
$K[\HK_{\Theta}]$. This is the starting point for the approach in
the present paper. We propose an entirely new structural approach
to Hecke--Kiselman monoids and their algebras.

We denote by $F$ the free monoid generated by the elements of the
set $X=\{x_1,\ldots, x_n\}$. However, the same notation will be
used for the generators of the monoid $\HK_\Theta$, if
unambiguous. For any words  $w, v\in F$ we say that $w$ is a
factor of $v$ if $v=v_1wv_2$ for some $v_1, v_2\in F$.

Let $C_n$ denote the  Hecke--Kiselman monoid associated to the
oriented cycle of length $n$, namely:   $x_1\rightarrow x_2
\rightarrow \cdots \rightarrow x_n\rightarrow x_1$. We consider
the deg-lex order on $F$ induced by $x_1<x_2<\cdots <x_n$. Let
$|w|_q$ denote the degree of $w\in F$ in the generator $x_q$. The
following result, proved in \cite{mecel_okninski2}, will be
crucial. We refer to \cite{diamond} for basic facts concerning
Gr\"obner bases and the diamond lemma.

\begin{theorem}\label{basisCn}
Let $\Theta = C_n$. Let $S$ be the system of reductions in $F$
consisting of all pairs of the form
\begin{enumerate}
    \item[(1)] $(x_ix_i,x_i)$ for all $i\in\{1,\dotsc,n\}$,
    \item[(2)] $(x_jx_i,x_ix_j)$ for all $i,j\in\{1,\dotsc,n\}$ such that $1<j-i<n-1$,
    \item[(3)] $(x_n(x_1\dotsm x_i)x_j,x_jx_n(x_1\dotsm x_i))$ for all
    $i,j\in\{1,\dotsc,n\}$ such that $i+1<j<n-1$,
    \item[(4)] $(x_iux_i,x_iu)$ for all $i\in\{1,\dotsc,n\}$ and $1\ne u\in F$
    such that $|u|_i=|u|_{i-1}=0$. Here, we
    write  $i-1 = n$ if $i = 1$,
    (we say, for the sake of simplicity, that the word $x_iux_i$ is of type $(4i)$),
    \item[(5)] $(x_ivx_i,vx_i)$ for all $i\in\{1,\dotsc,n\}$ and $1\ne v\in F$
    such that $|v|_i=|v|_{i+1}=0$. Here we write
     $i+1 = 1$ if $i =n$, (and similarly, we say that the word $x_ivx_i$ is of type $(5i)$).
\end{enumerate}
Then the set $\{w - v \mid \text{ for } (w, v) \in S\}$ is a
Gr\"obner basis of the algebra $K[C_n]$.
\end{theorem}

\begin{corollary}\label{podstawowy wniosek}
$C_n$ can be identified with the monoid $\mathcal{R}(S)$ of words
in $F$ that are reduced with respect to the system $S$, with the
operation defined for $u, w\in C_n$ by $u\cdot
w=\mathcal{R}_S(uw)$, where $\mathcal{R}_S(uw)$ is the
$S$--reduced form of the word $uw$. More precisely,
$\mathcal{R}(S)$ is the set of words in $F$ that do not have
factors of the form $w_{\sigma}$, where $\sigma=(w_{\sigma},
v_{\sigma})\in S$.
\end{corollary}

For $w,v\in F$, we write $w\arr{(\eta)} v$ in case $w=uw_\sigma z,
v=uv_\sigma z$ for some $u,z\in F$ and an element $(w_{\sigma},
v_{\sigma})$ of the set $S$ of reductions of type ($\eta$). Here
$(\eta)$ may be one of: ($1$) -- ($5$), or even more explicitly
($4i$) or ($5i$), for some $i$. More generally, $w\arr{(\eta)} v$
may also denote a sequence of consecutive reductions of type
($\eta$). If clear from the context, $w\rightarrow v$ will denote
an unspecified sequence of reductions.

Let $K\langle X\rangle=K\langle x_1,\ldots, x_n\rangle$ be the
free algebra over a field $K$.  The length of a word $w\in F$ is
denoted by $|w|$. By $\suff_m(w)$ ($\pref_m(w)$, respectively) we
mean the suffix (prefix, respectively) of length $m$ in $w$. For
every subset $Z\s F$ by $\suff(Z)$ ($\pref(Z)$, respectively) we
denote the set of all suffixes (prefixes, respectively) of
elements of $Z$. If $w\in F$ then $w^{\infty}$ denotes the
infinite word  $ww w\ldots$. For a subset $Z$ of a monoid $M$ by
$\langle Z \rangle$ we denote the submonoid generated by $Z$. If
$Z=\{w\}$, then we also write $\langle w\rangle$. If $S$ is a
semigroup, then $S^1$ stands for the monoid obtained by adjoining
an identity element to $S$. Also, $S^{0}$ denotes the semigroup
$S$ with zero adjoined. The support $\supp(\alpha)$ of an element
$\alpha =\sum_{i}m_is_i$, where $m_i\in K,s_i\in M$, of the
semigroup algebra $K[M]$ is defined as the set of all $s_{i}$ such
that $m_i\neq 0$. If $M$ has a zero element $\theta$ then the
contracted semigroup algebra $K_{0}[M]$ is defined as the factor
algebra $K[M]/K\theta$.

If $S$ is a semigroup, $A,B$ are nonempty sets and $P=(p_{ba})$ is
a $B\times A$ - matrix with entries in $S^{0}$, then the semigroup
of matrix type ${\mathcal M}^{0}(S,A,B;P)$ over $S$ is the set of
all triples $(s,a,b)$, where $s\in S, a\in A, b\in B$, with the
zero element $\theta$, with operation $(s,a,b)(s',a',b')=
(sp_{ba'}s',a,b')$ if $p_{ba'}\in S$ and $\theta $ otherwise.
Moreover ${\mathcal M}^{0}(K[S],A,B;P)$ denotes the corresponding
algebra of matrix type. It is defined as $K_{0}[{\mathcal
M}^{0}(S,A,B;P)]$ and (if $A,B$ are finite) it can be interpreted
as the set of all $A\times B$ - matrices over $K[S]$ with
operation $\alpha \beta = \alpha \circ P\circ \beta $, where
$\circ$ stands for the standard matrix product. This construction
plays a fundamental role in the structure and representations of
semigroup algebras. We refer to \cite{semalg}, Chapter~5, for
basic results. Structures of these types are crucial for the
approach and the results of this paper.

The material is organized as follows. In Section~\ref{form} we
derive the reduced form of almost all words in $F$ representing
the elements of the monoid $C_n$. The main results are stated in
Theorem~\ref{tw1} and Proposition~\ref{stw1}. In
Section~\ref{ideal} a natural ideal chain $I_{n-3}\subseteq
I_{n-4} \subseteq \cdots \subseteq I_{0}\subseteq I_{-1}$ of $C_n$
is introduced. And it is shown that all factors $I_{j-1}/I_{j}$,
and $(I_{n-3})^{0}$, have (modulo finitely many elements) a
structure of a semigroup $M_{j}$, $j=0,1,\ldots, n-2$, of matrix
type over an infinite cyclic semigroup. Using certain natural
involutions on $C_n$ (Definition~\ref{sigma-def}) it is shown that
these matrix structures have a very regular form, see
Corollary~\ref{wn3}, and Remark~\ref{symmetric}. The results of
this section are summarized in Corollary~\ref{struktura},  which
provides a very transparent structural tool for approaching the
monoids $C_n$ and their algebras. In Section~\ref{examples} these
results are illustrated with the (misleadingly simple) cases of
the monoids $C_n$ for $n=3$ and $n=4$. In
Section~\ref{Noetherian-chapter} it is first shown in
Theorem~\ref{semiprime} that all algebras of matrix type resulting
from the semigroups $M_j$ are prime. Then, in
Theorem~\ref{noether-cycle}, this is used to prove that $K[C_n]$
is a right and left Noetherian algebra. Existence of an embedding
of $K[C_n]$ into a matrix ring over a field follows, which is of
interest in the context of several earlier results on faithful
matrix representations of various classes of Hecke--Kiselman
monoids, \cite{type A},\cite{maz},\cite{OnK}.
Section~\ref{Noetherian-chapter} culminates with an important
application of the approach developed in this paper. Namely, a
characterization of all oriented graphs $\Theta$ such that the
corresponding Hecke--Kiselman algebra $K[\HK_{\Theta}]$ is
Noetherian, Theorem~\ref{Noetherian}. We prove that the latter is
equivalent to saying that each of the connected components of
$\Theta$ is either an oriented cycle or an acyclic graph. We
conclude with some open problems.

\section{The form of (almost all) reduced words in $C_n$}\label{form}

The main aim of this section is to prove that all elements of the
monoid $C_n$, except for finitely many words, have a very special
reduced form (with respect to the deg-lex order and the reduction
system $S$ introduced in Theorem~\ref{basisCn}). This will be the
key to describe the structure and properties of $C_n$ in the next
sections. Because, in view of Corollary~\ref{podstawowy wniosek},
we may identify the elements of $C_n$ with the reduced words in
$F$.

We adopt the following notation. If $i,j\in \{1,\ldots, n\}$ then
$x_i \cdots x_j$ denotes the product of all consecutive generators
from $x_i$ up to $x_j$ if $ i<j$, or down to $x_{j}$, if $i
> j$.

Let $q_i=x_1\cdots x_ix_{n-1}\cdots x_{i+1}\in F$, for
$i=0,\ldots, n-2$. Here we agree that $q_0 = x_{n-1} \cdots x_1$.
From Corollary~\ref{podstawowy wniosek} it follows that the word
$(x_nq_i)^k$ is reduced for every $k\geqslant0$.

 For every $i=0,\ldots, n-2$ we define two subsets $A_i$ and $B_i$ of $F$, as
 follows. First,
         \begin{equation*}
    A_i=\suff(\{(x_{k_s}\cdots x_s)(x_{k_{s+1}}\cdots x_{s+1})\cdots(x_{k_{i+1}}\cdots x_{i+1})\}),
    \end{equation*}
    where $s\in\{0, \ldots, i+1\}$, $k_{s+1}<k_{s+2}<\cdots <k_{i+1}\leqslant n-1$,
    $k_s\leqslant s$ and $k_q> q$ for $q=s+1, \ldots, i+1$.
    \\The convention is that the subset of $A_i$ corresponding to $s=i+1$ has the form
        $\suff(\{x_{k_{i+1}}\cdots x_{i+1}\})$, where $k_{i+1}\leqslant i+1$.
    Also, if $s=0$ then the corresponding subset of $A_i$  has the
    form\\
    $\suff (\{(x_{k_{s+1}}\cdots x_{s+1})\cdots(x_{k_{i+1}}\cdots x_{i+1})\})$,
    where $k_{s+1}<k_{s+2}<\cdots <k_{i+1}\leqslant n-1$ and $k_q> q$ for $q=s+1, \ldots, i+1$.\\
    The set $B_i$ is defined by
    \begin{equation*}
    B_i=\pref (\{x_n(x_{1}\cdots x_{i_1}x_{n-1}\cdots x_{j_1})
    \cdots x_n(x_{1}\cdots x_{i_r}x_{n-1}\cdots x_{j_r})x_nx_{n-1}\cdots x_{j_{r+1}}\}),
    \end{equation*}
    where $r\geqslant 0$, $i_{r}<i_{r-1}<\cdots< i_1<i+1$ and $i+1<j_1<j_2<\cdots <j_{r+1}\leqslant n$. \\
     Here, the subset of $B_i$ corresponding to $r=0$ has the form $\pref(\{x_nx_{n-1}\cdots x_{j_1}\})$.

The following result characterizes all reduced words that have a
factor of the form $x_nq_i$.

\begin{theorem}\label{tw1}
    \label{tw}
     Assume that $w$ is a reduced word that contains a factor of the form
    $x_nx_1\cdots x_i x_{n-1}\cdots x_{i+1}$ for some $i=0,\ldots, n-2$. Then
    $$w=a(x_nx_1\cdots x_i x_{n-1}\cdots x_{i+1})^kb$$
    for some $a\in A_i$, $b\in B_i$ and some $k\geq 1$.
    Moreover, all words of this type are reduced.
\end{theorem}

 We will use the following convention. By \textit{a block}
we mean a factor of the form $x_{k_{j}}\cdots x_{j}$, for some
$j\in\{s, \ldots, i+1\}$, appearing in the elements of the set
$A_i$ or a factor of the form $x_n(x_{1}\cdots
x_{i_k}x_{n-1}\cdots x_{j_k})$, for $k\in\{1, \ldots, r\}$, or
$x_nx_{n-1}\cdots x_{j_{r+1}}$, appearing in the elements of
$B_i$.

The main idea of the proof of Theorem~\ref{tw1} is to analyze the
possible forms of reduced words that satisfy certain additional
restrictions. The proof will be preceded by a series of technical
lemmas.
\begin{lemma}\label{lem1}
    If $w=x_{n-1}u$ is a reduced word, where $u\in F$ is such that $|u|_{n}=0$,
    then $w=x_{n-1}\cdots x_{k}$ for some $k\geqslant 1$.
\end{lemma}
\begin{proof} Let $w=x_{n-1}\cdots x_tu'$ for some $n-1\geqslant t>1$
and $\pref_1(u')=x_s$. If $s<t-1$, then the word $x_{n-1}\cdots
x_tx_s$ has a factor $x_tx_s$, with  $s<t-1$, so it is not
reduced. If $n> s \geqslant t$, then the word $x_{n-1}\cdots
x_tx_s$ has a factor $x_svx_s$, where $|v|_{s+1}=0$, whence it is
not reduced. Therefore, we get that $s=t-1$ and $w$ is a prefix of
the word of the form $x_{n-1}\cdots x_1 v$ for some $v$. Notice
that for every $k\neq n$ the word $x_{n-1}\cdots x_1x_k$ has a
factor $x_{k}zx_k$, where $|z|_{k+1}=0$, so it is not reduced. It
follows that $v$ must be the empty word. The assertion follows.
\end{proof}
\begin{lemma}\label{lem2}
    If  $w=x_{n}x_1u$ is a reduced word, where  $|u|_n=0$, then $w$ is of one of the forms
    \begin{enumerate}
        \item $w=x_nx_1x_{2}\cdots x_{i}x_{n-1}\cdots x_{j}$ for some $1\leqslant i<j\leqslant n-1$;
        \item $w=x_nx_1x_2\cdots x_{i}$ for some $1\leqslant i<n-1$.
    \end{enumerate}
\end{lemma}
\begin{proof}
Let $w=x_nx_1\cdots x_ku'$ for some $n-1>k\geqslant 1$ and let
$\pref_1(u')=x_s$ for some $s<n$. If $s<k$, then $w$ has a factor
$x_svx_s$, with $|v|_{s-1}=0$. On the other hand, if $n-1>s>k+1$
then we get a factor of the form $x_nx_1\cdots
x_kx_s\arr{(3)}x_sx_nx_1\cdots x_k$. It follows that
$\pref_1(u)\in\{x_{k+1}, x_{n-1}\}$. This means that
$w=x_nx_1\cdots x_ix_{n-1}v$ for some $i$ and some $v$ or
$w=x_nx_1\cdots x_i$ for some $i<n-1$. In the former case
Lemma~\ref{lem1} implies that $w=x_nx_1\cdots
x_ix_{n-1}x_{n-2}\cdots x_k$ for some $k<n$. If  $i\geqslant k$,
then $w$ has a factor $x_ix_{n-1}\cdots x_{i}$ and $|x_{n-1}\cdots
x_{i+1}|_{i-1}=0$, so $w$ is not a reduced word.
\end{proof}
\begin{lemma}\label{lem3}
    If $w=x_nux_n$ is a reduced word, where  $|u|_{n}=0$, then $u$ is of one of the
    forms
    \begin{enumerate}
        \item $u=x_{n-1}\cdots x_1$;
        \item $u=x_1\cdots x_ix_{n-1}\cdots x_j$ for $1\leqslant i<j\leqslant n-1$.
    \end{enumerate}
\end{lemma}
\begin{proof}
Notice that $\pref_1(u)\in\{x_1, x_{n-1}\}$, since otherwise if
$\pref_1(u)=j$, then $w$ has a factor  $x_nx_j$ for some
$1<j<n-1$, and thus it is not reduced.

Assume first that $\pref_1(u)=x_{n-1}$. Notice that
$|u|_1\geqslant 1$. If $|u|_1=0$, then $x_nux_n= x_nu$, so $w$ is
not reduced. By Lemma~\ref{lem1} it now follows that $w$ must be
of the form $x_n\cdots x_1 v x_n$ for some $v$. Then $v$ must be
the empty word, because for every $k$, if $\pref_1(v)=x_k$, then
$w$ has a factor of the form $x_kx_{k-1}\cdots x_1x_k$, where
$|x_{k-1}\cdots x_1|_{k+1}=0$, hence $w$ is not reduced.

Thus, assume that $\pref_1(u)=x_1$. Let $w=x_nx_1u'x_n$ for some
$u'$ such that $|u'|_n=0$. From Lemma~\ref{lem2} we know that
$w=x_nx_1\cdots x_ix_{n-1}\cdots x_jx_n$ for some $1\leqslant i<
j\leqslant n-1$ or $w=x_nx_1\cdots x_ix_n$ for some $1\leqslant
i<n-1$. However, in the latter case $|x_1\cdots x_i|_{n-1}=0$, so
that $w$ is not a reduced word.
\end{proof}

The following lemma shows that in the case where $i=0$ or $i=n-2$,
the reduced words with a factor $x_nq_i$ have an extremely simple
form.
\begin{lemma}\label{lem4}
    If a reduced word has a factor of the form $x_nx_1\cdots x_{n-1}$ or
    $x_nx_{n-1}\cdots x_1$, then it must be a factor of the infinite word
    $(x_nx_1\cdots x_{n-1})^{\infty}$ or  $(x_nx_{n-1}\cdots x_1)^{\infty}$.
\end{lemma}
\begin{proof} Define $x_{n+1}=x_1$ and $x_0=x_n$. Assume that
$w=ux_{k+1}x_{k+2}\cdots x_n x_1\cdots x_{k}v$ for some
$k=1,\ldots, n$ and some $u,v\in F$. We claim that
$\pref_1(v)=x_{k+1}$ and $\suff_1(u)=x_{k}$. Then the first part
of the assertion will follow.

If $\pref_1(v)=x_{s}$ for $s<k+1$, then the word $w$ has a factor
$x_sx_{s+1}\cdots x_{k}x_s$, whence it cannot be reduced.
Similarly, for $s>k+1$ the word $w$ has a factor of the form
$x_s\cdots x_{n}x_{1}\cdots x_{k}x_{s}$ and $|x_{s+1}\cdots
x_{n}x_{1}\cdots x_{k}|_{s-1}=0$, which is not possible. Suppose
that $\suff_1(u)=x_{s}$ for some  $s\neq k$. If $s<k$, then $w$
has a factor $x_{s}x_{k+1}\cdots x_nx_1\cdots x_{s-1} x_{s}$ and
$|x_{k+1}\cdots x_nx_1\cdots x_{s-1}|_{s+1}=0$. Similarly, if
$s>k$, then $w$ has a factor $x_sx_{k+1}\cdots x_s$, where
$|x_{k+1}\cdots x_{s-1}|_{s+1}=0$, which leads to a contradiction
again. This proves the claim.

The second part of the lemma follows by a symmetric argument.
\end{proof}

The next few lemmas will be used to determine the desired shape of
the elements of $B_i$, which are the endings of the considered
class of reduced words.

\begin{lemma}\label{lem5}
    Let $w\in F$ be of the form $w=x_nx_1\cdots x_{i_1}x_{n-1}\cdots x_{j_1}x_nx_1\cdots x_{i_2}x_{n-1}\cdots x_{j_2}$,
    where $1\leqslant i_p<j_p\leqslant n-1$ for $p=1, 2$. If $w$ is a reduced word, then
    \begin{enumerate}
        \item $i_2<j_1$,
        \item $i_1\geqslant i_2$,
        \item $j_1\leqslant j_2$.
    \end{enumerate}
\end{lemma}
\begin{proof} Suppose that $i_2\geqslant j_1$. Then $j_1<n-1$.
Thus, $w$ has a  factor of the form $x_{j_1}x_nx_1\cdots
x_{j_1-1}x_{j_{1}}$ and $|x_nx_1\cdots x_{j_1-1}|_{j_1+1}=0$,
whence $w$ is not in a reduced form.

 Suppose that $i_1<i_2$. Then
$i_1<n-2$ and $w$ has a factor of the form $x_{i_1}x_{n-1}\cdots
x_{j_1}x_nx_1\cdots x_{i_1}$ and $|x_{n-1}\cdots
x_{j_1}x_nx_1\cdots x_{i_{1}-1}|_{i_{1}+1}=0$, because we know
that $i_1+1<i_2+1\leqslant j_1$, so that $i_1+1<j_1$. Therefore,
again $w$ cannot be in the reduced form.

 Suppose that $j_1>j_2$.
Then the hypothesis and the first part of the lemma imply that
$i_2\leqslant j_2-1<j_1-1$. Notice that in this case $w$ has a
factor of the form $x_{j_1}x_nx_1\cdots x_{i_2}x_{n-1}\cdots
x_{j_1+1}x_{j_1}$ and $|x_nx_1\cdots x_{i_2}x_{n-1}\cdots
x_{j_1+1}|_{j_1-1}=0$, because $i_2<j_1-1$; so $w$ is not in the
reduced form.
\end{proof}
\begin{lemma}\label{lem6}
    Let $w$ be a reduced word such that
    \begin{equation*}
    w=x_nx_1\cdots x_{i_1}x_{n-1}\cdots x_{j_1}x_nx_1\cdots x_{i_2}x_{n-1}\cdots x_{j_2}.
    \end{equation*}
    If $i_1=i_2=i$, then $j_1=i+1$ and if  $j=j_1=j_2$, then $i_1=i_2=j-1$.
    Moreover, if  $i_1+1=j_1$ and $i_2+1<j_2$, then $j_2>j_1$.
\end{lemma}
\begin{proof} Assume that $i_1=i_2=i$. Then $w$ contains a reduced factor
$x_{i}x_{n-1}\cdots x_{j_1}x_nx_1\cdots x_{i}$. Hence,
$|x_{n-1}\cdots x_{j_1}x_nx_1\cdots x_{i-1}|_{i+1}\geqslant 1$.
Since $i<j_1$, this implies that $j_1=i+1$.

Consider the case where $j=j_1=j_2$. Then $w$ contains a reduced
factor  $x_jx_nx_1\cdots x_{i_2}x_{n-1}\cdots x_j$. This implies
that $|x_nx_1\cdots x_{i_2}x_{n-1}\cdots x_{j+1}|_{j-1}\geqslant
1$. Since $i_2<j$, we must have $i_2=j-1$. By Lemma~\ref{lem5}
$j-1=i_2\leqslant i_1<j$, so that $i_1=i_2=j-1$.

Finally, assume that $i_1+1=j_1$ and $i_2+1<j_2$. From
Lemma~\ref{lem5} we know that $j_2\geqslant j_1$. Suppose that
$j_2=j_1$. Then the part of the statement that has already been
proved implies that $i_2=i_1=j_1-1=j_2-1$, which contradicts the
hypothesis. This completes the proof.
\end{proof}
\begin{corollary}\label{wn1}
    If a reduced word $w$ is of the form
    \begin{equation*}
    ux_n(x_{1}\cdots x_{i_1}x_{n-1}\cdots x_{j_1})x_n(x_{1}\cdots x_{i_2}x_{n-1}\cdots x_{j_2})\cdots x_n(x_{1}\cdots x_{i_r}x_{n-1}\cdots x_{j_r})x_nv
    \end{equation*}
    for some $u, v$ such that $|u|_{n}=|v|_{n}=0$, then it follows that
    \begin{equation*}
    i_{r}\leqslant i_{r-1}\leqslant\cdots\leqslant i_1< j_1\leqslant j_2\leqslant\cdots \leqslant j_{r}.
    \end{equation*}
    Furthermore, if  $i_k+1=j_k$ for some  $k$, then $i_k=i_s$ and $j_k=j_s$ for $s=1,\ldots, k-1$.\\
    Moreover, if for some $l$ we have
        $i_l+1<j_l$, then $i_{r}< i_{r-1}<\cdots<i_l< i_l+1<j_l<\cdots< j_{r}$.
    \\If $l>1$, then also $i_l\leqslant i_{l-1}<j_{l-1}<j_l$.
\end{corollary}
\begin{proof}
The inequalities of the first part of the assertion follow
directly from Lemma~\ref{lem5}.

Notice that if $i_k+1=j_k$, then Lemma~\ref{lem5} implies that
$i_k\leqslant i_{k-1}<j_{k-1}\leqslant j_k$, so that $i_k=i_{k-1}$
and $j_k=j_{k-1}$. Repeating this argument, we get that if
$i_k+1=j_k$ for some $k$, then $i_k=i_s$ and $j_k=j_s$ for
$s=1,\ldots, k-1$.

Suppose that for some  $l$ we have  $i_l+1<j_l$. If
$i_{l+1}=i_{l}$ or $j_{l+1}=j_{l}$, then Lemma~\ref{lem6} implies
that $j_{l+1}=i_l+1$, contradicting the hypothesis. Now, repeating
this argument and using the part of the assertion that has already
been proved we get that $i_{r}< i_{r-1}<\cdots<i_l$ and
$j_l<\cdots< j_{r}$.

Finally, assume that $l>1$ and also $i_l+1<j_l$. Then, by
Lemma~\ref{lem5} we know that $i_l\leqslant
i_{l-1}<j_{l-1}\leqslant j_l$. Suppose that $j_{l-1}=j_l$. Then
Lemma~\ref{lem6} implies that $i_l=i_{l-1}=j_{l}-1$, a
contradiction. Thus $j_{l-1}<j_l$. This completes the proof.
\end{proof}

The next few lemmas will be used to deal with the shape of the
elements of the set $A_i$, which are the beginnings of the
considered class of reduced words.

\begin{lemma}\label{lem7}
    Let $w$ be a reduced word such that
    \begin{equation*}
    w=v(x_{k_s}\cdots x_s)ux_{n}x_1\cdots x_{i}x_{n-1}\cdots x_{i+1},
    \end{equation*}
    where $i=1,\ldots, n-3$, $k_{s}<s\leqslant i+1$, $|u|_{j}=0$ for $j=1,\ldots, s$ and $|u|_n=|v|_n=0$.
    Then $v=x_{r}x_{r+1}\cdots x_{k_{s}-1}$ for some $r\geqslant 1$.
\end{lemma}
\begin{proof}  Assume that $k_s>1$ and $x_j=\suff_1(v)$. If $j>k_{s}+1$, then $w$ has a factor
$x_jx_{k_s}$, which is not reduced. Similarly, if  $j<k_{s}-1<i$,
then $w$ has a factor $x_j(x_{k_s}\cdots x_s)u x_nx_1\cdots
x_{j}$. Since $|(x_{k_s}\cdots x_s)u x_n x_1\cdots
x_{j-1}|_{j+1}=0$, this also leads to a contradiction. It follows
that $j=k_{s}-1$.

Suppose that  $w=v'x_1\cdots x_sux_nx_1\cdots x_{i}x_{n-1}\cdots
x_{i+1}$ for some nonempty word $v'$. Let $x_p=\suff_1(v')$ for
some $p\neq n$. If $p\neq 2$, then  $w$ has a factor  $x_px_1$,
that is not reduced. Moreover, for $p=2$, since $1\leqslant
k_{s}<s$, and hence $2\leqslant s$, we get a factor of the form
$x_2x_1x_2$, that is not reduced. This contradiction completes the
proof.
\end{proof}
\begin{lemma}\label{lem8}
    Assume that a reduced word $w$ is of the form
    \begin{equation*}w=u(x_{k_s}\cdots x_s)(x_{k_{s+1}}\cdots x_{s+1})\cdots(x_{k_{i+1}}\cdots
    x_{i+1})x_nx_1\cdots x_ix_{n-1}\cdots x_{i+1}\end{equation*}
    where $s\leqslant i+1$, $k_s\geqslant s$, $k_q>q$ for $q=s+1,\ldots, i+1$, $|u|_n=0$.
    Then $\suff_1(u)\in\{x_{k_s+1}, x_{s-1}\}$.
\end{lemma}
\begin{proof} Let $\suff_1(u)=x_j$ for some $j\neq n$. We consider the following cases:
\begin{itemize}
    \item if  $j>k_s+1$ or $j=k_s$, then  $w$ has a factor  $x_jx_{k_s}$, which is not reduced,
    \item if $j=k_s-1$ and $k_s>s$, then  $w$ has a factor  $x_{j}x_{j+1}x_{j}$, which is not reduced,
    \item if  $s \leqslant j<k_s-1\leqslant i$, then we have a factor  $x_jx_{k_s}\cdots x_{j+1}x_j$ such that
    $|x_{k_s}\cdots x_{j+1}|_{j-1}=0$; whence $w$ is not a reduced word,
    \item if $j<s-1$, then $w$ has a factor of the form
    \begin{equation*}x_j(x_{k_s}\cdots x_s)(x_{k_{s+1}}\cdots x_{s+1})\cdots(x_{k_{i+1}}\cdots
    x_{i+1})x_nx_1\cdots x_{j-1}x_j=x_jzx_j\end{equation*}
    where $|z|_{j+1}=0$; so $w$ is not a reduced word.
\end{itemize}
We have thus proved that $\suff_1(u')\in\{x_{k_s+1}, x_{s-1}\}$,
as desired.
\end{proof}
\begin{lemma}\label{lem9}
    Assume that a reduced word $w$ is of the form
    \begin{equation*} w=ux_nx_1\cdots x_ix_{n-1}\cdots x_{i+1}
    \end{equation*}
    for some $i\in\{1,\ldots, n-3\}$, where $|u|_{n}=0$ and $u$ has no factors of the form $x_l\cdots x_j$, where $l<j$.
    Then $u$ is of the form
    \begin{equation*}
    u=(x_{k_s}\cdots x_s)(x_{k_{s+1}}\cdots x_{s+1})\cdots(x_{k_{i+1}}\cdots x_{i+1}),
    \end{equation*}
    where $s\leqslant i+1$, $k_s<k_{s+1}<\cdots <k_{i+1}\leqslant n-1$, $k_s\geqslant s$, $k_q> q$ for $q=s+1, \ldots, i+1$.
\end{lemma}
\begin{proof}
Let $x_s=\suff_1(u)$ for some $s\neq n$. If $s>i+1$, then $w$ has
a factor $x_{s}x_nx_1\cdots x_ix_{n-1}\cdots x_{s+1}x_s$ and
$|x_nx_1\cdots x_ix_{n-1}\cdots x_{s+1}|_{s-1}=0$, whence $w$ is
not reduced. If $s<i$, then $w$ has a factor $x_{s}x_nx_1\cdots
x_{s-1}x_{s}$, so it is not reduced. It follows that
$\suff_1(u)=x_{i+1}$.

Applying Lemma~\ref{lem8}, we deduce that if $|u|>1$, then $w=u'
x_{i+2}x_{i+1}x_nx_1\cdots x_ix_{n-1}\cdots x_{i+1}$ or $w=u'
x_{i}x_{i+1}x_nx_1\cdots x_ix_{n-1}\cdots x_{i+1}$ for some $u'\in
F $. In the latter case, $w$ has a factor $x_ix_{i+1}$, which
contradicts the hypothesis. Repeating this argument we get the
desired form of $u$, because an inequality $k_j \geqslant k_{j+1}$
would allow a reduction of type (5).
\end{proof}

\begin{lemma}\label{lem10}
    If $w$ is a reduced word and $w=ux_nx_{1}\cdots x_i x_{n-1}\cdots x_{i+1}$
         for some
    $i=1,\ldots, n-3$ and some $u\in F$ such that $|u|_{n}=0$, then $u$ is of the form
    \begin{equation*}
    u=(x_{k_s}\cdots x_s)(x_{k_{s+1}}\cdots x_{s+1})\cdots(x_{k_{i+1}}\cdots x_{i+1}),
    \end{equation*}
    where $s\leqslant i+1$, $k_{s+1}<k_{s+2}<\cdots <k_{i+1}\leqslant n-1$ and $k_q> q$ for $q=s+1, \ldots, i+1$
    (but perhaps $k_s\leqslant s$).
\end{lemma}
\begin{proof}
If $u$ has no factors of the form $x_lx_{l+1}\cdots x_j$, where
$l<j$, then the desired form of the word follows from
Lemma~\ref{lem9}.

Thus, assume otherwise. Let
\begin{equation*}
w=zvx_nx_1\cdots x_i x_{n-1}\cdots x_{i+1},
\end{equation*}
where the word $v$ does not have increasing factors,
$\suff_2(z)=x_{j-1}x_{j}$ for some $j$ and $\pref_1(v)=x_{k_s}$
for $k_s\neq j+1$. Then Lemma~\ref{lem9} implies that
\begin{equation*}v=(x_{k_{s}}\cdots x_{s})\cdots(x_{k_{i+1}}\cdots x_{i+1})
\end{equation*}
and $s\leqslant i+1$, $k_{s}<k_{s+1}\cdots <k_{i+1}\leqslant n-1$,
$k_q>q$ for $q=s+1,\ldots, i+1$, $k_{s}\geqslant s$.

Moreover, by Lemma~\ref{lem8}, we have $x_j\in\{x_{k_s+1},
x_{s-1}\}$. If $j=k_s+1$, then $w$ has a factor
$x_{j-1}x_{j}x_{j-1}$, which is impossible. Hence $j=s-1$, and
then
\begin{equation*}|(x_{k_s}\cdots x_s)(x_{k_{s+1}}\cdots x_{s+1})\cdots(x_{k_{i+1}}\cdots x_{i+1})|_p=0\end{equation*}
for $p=1, \ldots, s-1$, and the assertion follows from
Lemma~\ref{lem7}.
\end{proof}

\noindent\textbf{Proof of Theorem~\ref{tw1}.} By
Corollary~\ref{podstawowy wniosek} it is clear that all words
described in the statement are reduced.

Let $w\in F$ be a reduced word that contains a factor $x_nq_i$. By
Lemma~\ref{lem4} the assertion holds for $i=0, n-2$. Notice that
if the word $w$ has the form $x_nx_1\cdots x_ix_{n-1}\cdots
x_{i+1}v$ for some $v\in C_n$, then we must have $\pref(v)=x_n$.
Indeed, if $\pref(v)=x_j$ for $j\leqslant i$, then
$x_nq_ix_j\arr{(4j)}x_nq_i$. Similarly, if $i+1\leqslant
j\leqslant n-1$, then $x_nq_ix_j\arr{(5j)}x_nx_1\cdots
x_ix_{n-1}\cdots x_{j+1}x_{j-1}\cdots x_{i+1}$.

From Lemma~\ref{lem3} and Lemma~\ref{lem5} we know that if
$i=1,\ldots, n-3$ then $w$ is of the form
\begin{equation*}
u(x_nx_1\cdots x_{i_1}x_{n-1}\cdots x_{j_1})\cdots (x_nx_1\cdots
x_{i_m}x_{n-1}\cdots x_{j_m})x_nv,
\end{equation*}
for some $m$, where $1\leqslant i_k<j_k\leqslant n-1$ for every
$k$ and $|u|_n=|v|_n=0$.

In view of Corollary~\ref{wn1} this implies that $w$ is of the
form
\begin{equation} \label{ux}
u(x_nx_{1}\cdots x_{i}x_{n-1}\cdots x_{i+1})^{k}x_n(x_{1}\cdots
x_{i_1}x_{n-1}\cdots x_{j_1}) \cdots x_n(x_{1}\cdots
x_{i_r}x_{n-1}\cdots x_{j_r})x_nv,
\end{equation}
where $i_{r}<i_{r-1}<\cdots<i_1<i+1$ and $i+1<j_1<j_2<\cdots
<j_{r}$ and $|u|_{n}=|v|_{n}=0$, where the factor of the form
$(x_{1}\cdots x_{i_1}x_{n-1}\cdots x_{j_1})\cdots x_n(x_{1}\cdots
x_{i_r}x_{n-1}\cdots x_{j_r})$ does not have to occur in $w$ (that
is, $w=u(x_nx_{1}\cdots x_{i}x_{n-1}\cdots x_{i+1})^{k}x_nv$) and
then we put $r=0$.

Notice that $\pref_1(v)\in\{x_1, x_{n-1}\}$, since otherwise $w$
contains a factor $x_nx_s$ for $s<n$, which is not reduced.

If $\pref_1(v)=x_{n-1}$, Lemma~\ref{lem1} implies that
$v=x_{n-1}\cdots x_{j_{r+1}}$. Moreover, we must have
$j_r<j_{r+1}$, as otherwise $w$ has a factor $x_{j_{r}}x_n\cdots
x_{j_r+1}x_{j_r}$ such that $|x_n\cdots x_{j_r+1}|_{j_r-1}=0$,
which is not reduced.

If $\pref_1(v)=x_{1}$, then by Lemma~\ref{lem2} and
Corollary~\ref{wn1} we get $v=x_1\cdots x_{i_{r+1}}x_{n-1}\cdots
x_{j_{r+1}}$ for $i_{r+1}<i_r$ and $j_{r+1}> j_{r}$, if $r>0$. If
$r=0$, then in view of (\ref{ux}) we have $w=pq$, where $p\in F$
and
\begin{equation*}q=x_nx_1\cdots x_ix_{n-1}\cdots x_{i+1}x_nx_1\cdots x_{i_{1}}x_{n-1}\cdots x_{j_{1}}\end{equation*}
Corollary~\ref{wn1} implies that $i_1\leqslant i$ and
$j_1\geqslant i+1$. The desired form of the elements of the set
$B_i$ follows.

Since $k\geqslant 1$, the desired form of the elements of the set $A_i$ follows by Lemma~\ref{lem10}.
This completes the proof of Theorem~\ref{tw1}. $\square$\\

Our next aim is to show that Theorem~\ref{tw1} characterizes
reduced forms of almost all elements of $C_n$.
\begin{defi} \label{matrix struct}
    For every $i=0,\ldots, n-2$ we denote by $\tilde{M_i}$ the following set
    \begin{equation}
    \tilde{M_i}=\{a(x_nq_i)^kb\in C_n: a\in A_i, b\in B_i, k\geqslant 1\}.
    \end{equation}
(the set of reduced forms of elements of $C_n$ that have a factor
$x_nq_i$). Define also $\tilde{M}=\bigcup_{i=0}^{n-2}\tilde{M_i}$.
\end{defi}

Corollary~\ref{podstawowy wniosek} ensures that two elements  $w,
w'\in\tilde{M}$ are equal in $C_n$ if and only if the equality
$w=w'$ holds in the free monoid $F$ generated by $x_1, \ldots,
x_n$. In particular, we can write $\tilde{M}\s C_n$. This
identification will be often used without further comment.

\begin{proposition}\label{stw1}
    $C_n\setminus \tilde{M}$ is a finite set.
\end{proposition}
\begin{proof} Let $w$ be a reduced word that does not contain any factor
of the form $x_nx_1\cdots x_ix_{n-1}\cdots x_{i+1}$. Assume that
$|w|_n=k$, so that $w=u_0x_nu_1x_n\cdots x_nu_k$ for some words
$u_j$, $j=0,\ldots, k$, such that $|u_j|_n=0$.

If $k=0$, then $|w|_1\leqslant1$, since otherwise $w$ contains a
factor $x_1yx_1$, where $|y|_{n}=0$. We will prove by induction on
$m=1,\ldots, n-1$, that $|w|_{m}\leqslant m$. We know that the
assertion holds for $m=1$. Suppose that $|w|_{m-1}\leqslant m-1$
and also that $|w|_m>m$. Then $w$ has a factor of the form
$x_{m}yx_{m}$, where $|y|_{m-1}=0$ (because $|w|_{m-1}\leqslant
m-1$), whence $w$ is not reduced, a contradiction. This proves the
inductive assertion. It follows that if $|w|_n=0$ then $|w|\leq
1+\cdots+ (n-1)=\frac{n(n-1)}{2}=C_1(n)$. So the number of such
possible words does not exceed $n^{C_1(n)}$.

Assume now that $|w|_n=k\geqslant 1$. By Lemma~\ref{lem3} for
$s=1, \ldots, k-1$ we get $u_s= x_1\cdots x_{i_s}x_{n-1}\cdots
x_{j_s}$ for some $1\leqslant i_s<j_s\leqslant n-1$. Hence, in
particular it follows that the length of every $u_s$, for
$s=1,\ldots, k-1$, is at most $n-1$. Moreover, the assumption
implies that $u_1=x_1\cdots x_{i_1}x_{n-1}\cdots x_{j_1}$ where
$i_1+1<j_1$. From Corollary~\ref{wn1} it follows that we have the
following inequalities: $1\leqslant i_{k-1}< i_{k-2}<\cdots<
i_1+1<j_1<\cdots< j_{k-1}\leqslant n-1$. In particular, this
implies that $k\leqslant\frac{n}{2}+1$, and hence the length of
the reduced word $v=x_nu_1x_n\cdots x_n$ (such that $w=u_0vu_k$)
is at most $k+(k-1)(n-1)<kn\leqslant (\frac{n}{2}+1)n=C_2(n)$. By
the first part of the proof we get that that the length of $u_0$
and of $u_k$ is at most $C_1(n)$. We have proved that every
reduced word of the form  $w=u_0x_nu_1x_n\cdots x_nu_k$ has length
at most $C_2(n)+2C_1(n)$; so there are at most
$(n+1)^{C_2(n)+2C_1(n)}$ such words.

Consequently, the cardinality of $C_n\setminus\tilde{M}$ is
bounded by $n^{C_1(n)}+ (n+1)^{C_2(n)+2C_1(n)}$, so it is finite.
\end{proof}

\section{An ideal chain and related structures of matrix type}\label{ideal}

Our next goal is to introduce a special ideal chain in the monoid
$C_{n}$ that is strongly related to certain structures of matrix
type. This will be essential when dealing with the structure and
properties of the algebra $K[C_n]$, and consequently of every
Hecke--Kiselman algebra, in Section~\ref{Noetherian-chapter}.

In view of Corollary~\ref{podstawowy wniosek} we identify elements
of $C_n$ with the (unique) reduced forms of words in $F$.

\subsection{An ideal chain}\label{s: idealy}
We will introduce two special families of ideals of  $C_n$. First,
let
$$I_i=\{w\in C_n: C_nwC_n\cap\langle
x_nq_i\rangle=\emptyset\}$$ for $i=0,\ldots, n-2$.  We also define
$I_{-1}=I_0\cup C_n x_nq_{0}C_n$. It is clear that every $I_i$ is
an ideal in $C_n$, if it is nonempty. We show that
$I_{n-2}=\emptyset$. This is a consequence of the following
observation.
\begin{lemma}\label {lem o mno}
    Let $w\in C_n$. If $k=1,\ldots, n$ then the reduced form of
    $(x_{k+1}\cdots x_nx_1\cdots x_{k})w$
    is a factor of the infinite word $(x_nx_1\cdots x_{n-1})^{\infty}$.
    Moreover $(x_{k+1}\cdots x_nx_1\cdots x_{k})w$ has a prefix of the form $x_{k+1}\cdots x_nx_1\cdots x_{k}$.
\end{lemma}
\begin{proof}
Let $w\in C_n$.  We proceed by induction on the length of $w$. The
assertion is clear for the empty word. Assume that $|w|>1$. Let
$w=x_jw'$ for some $j$. If $j=k+1$ (for $k=n$ we put $j=1$), then
the assertion follows by the induction hypothesis for $k+1$ and
$w'$. If $j=k$, the assertion is also clear by induction.
Otherwise,
\begin{equation*}
(x_{k+1}\cdots x_nx_1\cdots x_{k})x_jw'\arr{(4j)}(x_{k+1}\cdots
x_nx_1\cdots x_{k})w'
\end{equation*}
and the assertion again follows by the induction hypothesis.
\end{proof}
\begin{corollary}\label{wn5}
    $I_{n-2}=\emptyset$
\end{corollary}
\begin{proof} Let $w\in C_n$. By Lemma~\ref{lem o mno}, $(x_nx_1\cdots x_{n-1})w$ is a factor of $(x_nx_1\cdots x_{n-1})^{\infty}$.
 Hence, in particular
$(x_nx_1\cdots x_{n-1})w\notin I_{n-2}$, so that $w\notin
I_{n-2}$. The assertion follows.
\end{proof}
 A dual version of Lemma~\ref{lem o mno} also holds.
 In order to prove this, we introduce a natural involution of the monoid $C_n$ that will be useful also later.
\begin{defi} \label{involut}
    Let $\tau: \langle x_1,\ldots, x_n\rangle\longrightarrow \langle x_1,\ldots, x_n\rangle$ be the involution such that
    \begin{equation*}
    \tau(x_i)=\begin{cases}x_{n-i} & \textrm{ for }i\neq n\\ x_n & \textrm{ for }i=n \end{cases}.
    \end{equation*}
    It is easy to see that $\tau$ preserves the set of defining relations of $C_n$. Hence,
    it determines an involution of $C_n$, also denoted by $\tau$.
    \end{defi}

\begin{lemma}\label{lem o mno dual}
    Let $w\in C_n$. If $k=0,\ldots, n-1$ then the reduced form of
    $w(x_{k+1}\cdots x_nx_1\cdots x_{k})$ is a factor of the infinite word
     $(x_nx_1\cdots x_{n-1})^{\infty}$.
    Moreover, $w(x_{k+1}\cdots x_nx_1\cdots x_{k})$ has a suffix of the form $x_{k+1}\cdots x_nx_1\cdots x_{k}$.
\end{lemma}
\begin{proof}
 Notice that for every $k\in\{0,\ldots, n-1\}$ we have
$\tau(x_{k+1}\cdots x_nx_1\cdots x_{k})=x_{n-k}\cdots x_nx_1\cdots
x_{n-k-1}$, where for $k=0$ we put  $x_{k+1}\cdots x_nx_1\cdots
x_{k}=x_1\cdots x_n$. In particular, the image under  $\tau$ of
every factor of the infinite word $(x_nx_1\cdots
x_{n-1})^{\infty}$ is a factor of this word. Let $w\in C_n$. The
above shows that $\tau(w(x_{k+1}\cdots x_nx_1\cdots
x_{k}))=x_{n-k}\cdots x_nx_1\cdots x_{n-k-1}\tau(w)$.
Lemma~\ref{lem o mno} yields that $\tau(w(x_{k+1}\cdots
x_nx_1\cdots x_{k}))$ is a factor of $(x_nx_{1}\cdots
x_{n-1})^{\infty}$, with a prefix of the form $x_{n-k}\cdots
x_nx_1\cdots x_{n-k-1}$. Applying the involution $\tau$ again we
get that $w(x_{k+1}\cdots x_nx_1\cdots x_{k})$ is a factor of
$(x_nx_{1}\cdots x_{n-1})^{\infty}$, with a suffix of the form
$x_{k+1}\cdots x_nx_1\cdots x_{k} = \tau(x_{n-k}\cdots
x_nx_1\cdots x_{n-k-1})$. The assertion follows.
\end{proof}

We can now return to the ideal chain $I_j$.

\begin{lemma}
$I_{i+1}\subseteq I_i$ for $i=0,\ldots, n-3$.
\end{lemma}
 \begin{proof} We claim that for every $i=0,\ldots, n-3$ and $l\geqslant1$ we
have $(x_nq_i)^l\notin I_{i+1}$; in other words, there exist $u,
v\in C_n$ such that $u(x_nq_i)^lv\in\langle x_nq_{i+1}\rangle$.

By Corollary~\ref{wn5}, the assertion holds for $i=n-3$. So,
assume that $i\in\{0, \ldots, n-4\}$.  First, notice that for
every $i+2\leqslant m\leqslant n-1$ we have
\begin{equation*}
(x_nx_1\cdots x_ix_{n-1}\cdots x_{i+1})x_m\cdots
x_{i+2}\arr{(5m)}x_nx_1\cdots x_ix_{n-1}\cdots
x_{m+1}x_{m-1}\cdots x_{i+1}x_m\cdots x_{i+2}.
\end{equation*}
Using the inequality $i+1<i+2<n-1$ we conclude that for $i\neq
n-3$
\begin{align*}
(x_nq_i)x_m\cdots x_{i+2}&\arr{(2)}x_n(x_1\cdots x_{i})x_{m-1}\cdots{x_{i+1}}x_{n-1}\cdots x_{i+2}\\
&\arr{(3)}(x_{m-1}\cdots x_{i+2})x_nx_1\cdots x_{i+1}x_{n-1}\cdots
x_{i+2}.
\end{align*}
It follows that for $1<k\leqslant n-i-1$
\begin{align*}
(x_nq_i)^kx_{i+2}&\arr{(5 (i+2))}
(x_nq_i)^{k-1}x_nx_1\cdots x_ix_{n-1}\cdots x_{i+3}x_{i+1}x_{i+2}\\
&\arr{(2)}(x_nq_i)^{k-1}x_nx_1\cdots x_{i+1}x_{n-1}\cdots x_{i+2}\\
&\arr{(5 (i+1))}(x_nq_i)^{k-2}x_nx_1\cdots x_ix_{n-1}\cdots x_{i+2}x_nx_1\cdots x_{i+1}x_{n-1}\cdots x_{i+2}\\
&\arr{(5i)}(x_nq_i)^{k-2}x_nx_1\cdots x_{i-1}x_{n-1}\cdots x_{i+2}x_nx_1\cdots x_{i+1}x_{n-1}\cdots x_{i+2}\\
&\arr{(5(i-1))}\arr{(5(i-2))}\cdots\arr{(51)}\arr{(5n)}(x_nq_i)^{k-2}x_{n-1}\cdots x_{i+2}x_nx_1\cdots x_{i+1}x_{n-1}\cdots x_{i+2}\\
&\arr{(*)} x_{n-k+1}\cdots x_{i+2}(x_nx_1\cdots
x_{i+1}x_{n-1}\cdots x_{i+2})^{k-1},\end{align*}
where $\arr{(*)}$ stands for applying $k-2$ times the observation made before. \\
Moreover, the first two steps of the above computation show that
for $k=1$ we get
\begin{align*}
(x_nq_i)^kx_{i+2}\arr{}(x_nq_{i+1})^k.
\end{align*}
Similarly, for every positive integer $m$, if $m-1=l(n-i-1)+k$,
where $l\geqslant0$ and $n-i-2\geqslant k\geqslant0$, then
\begin{equation*}(x_nx_1\cdots x_ix_{n-1}\cdots x_{i+1})^mx_{i+2}\arr{}\begin{cases}x_{n-k}\cdots x_{i+2}
(x_nx_1\cdots x_{i+1}x_{n-1}\cdots x_{i+2})^{m-l-1}\textrm{ for }k\neq 0\\
(x_nq_{i+1})^{m-l}\textrm{ for }k=0.
\end{cases}\end{equation*}
Hence, for every $i$ and $m-1=l(n-i-1)+k$, defining
$u=x_nx_1\cdots x_{i+1}x_{n-1}\cdots x_{n-k+1}$ if $k\neq 0$, and
$u=1$ if $k=0$, we get $u(x_nq_i)^mx_{i+2}\in \langle
x_nq_{i+1}\rangle$. Thus,  $(x_nq_{i})^m\notin I_{i+1}$, which
proves our claim.

Let $i=0,\ldots, n-3$. Suppose that there exists $w\in
I_{i+1}\setminus I_i$.
 Let  $x, y\in C_n$ and $k\geqslant1$ be such that  $xwy=(x_nq_i)^k$.
 By the first part of the proof, there exist $u, v\in C_n$ such that
 $uxwyv\in \langle x_nq_{i+1}\rangle$. Thus, $w\notin I_{i+1}$, which leads to a contradiction.
 \end{proof}

As a consequence, we get the following ideal chain in $C_n$
\begin{equation*} I_{n-3}\vartriangleleft  \cdots \vartriangleleft I_0 \vartriangleleft I_{-1} \vartriangleleft C_n.
\end{equation*}

In order to prove that certain elements of $C_n$ are contained in
$I_i$ we introduce another useful family of ideals $Q_i$. Let
$\textrm{Map}(\mathbb{Z}^n, \mathbb{Z}^n)$ denote the monoid of
all functions $\mathbb{Z}^n\longrightarrow \mathbb{Z}^n$, under
composition. Consider the homomorphism $f: C_n\longrightarrow
\textrm{Map}(\mathbb{Z}^n, \mathbb{Z}^n)$, introduced in
\cite{ara}, which is defined on generators $x_{i}$ of $C_{n}$ as
follows.
\begin{equation*}
f(x_i)(m_1, \ldots, m_n)= \begin{cases} (m_1, \ldots m_{i-1},
m_{i+1}, m_{i+1},\ldots, m_{n})\textrm{ for }i\neq n\\(m_1,
\ldots, m_{n-1}, m_1+1)\textrm{ for }i=n .\end{cases}
\end{equation*}
If $w\in C_n$ then the components of $f(w)(m_1, \ldots, m_n)$ are
polynomials in the variables $m_1,\ldots, m_n$. Let $\supp (f(w))$
be the minimal subset $N$ of the set $M=\{1, \ldots, n\}$ such
that for every $(m_1,\ldots, m_n)\in\mathbb{Z}^n$ the components
of $f(w)(m_1,\ldots, m_n)$ are polynomials depending on the
variables with indices from the set $N$. So $|\supp (f(w))|$
denotes the number of variables on which the value of $f(w)$
depends. For example, if  $f(w)(m_1,\ldots, m_n)=(m_1, \ldots,
m_{i-1}, m_{i+1}, m_{i+1},\ldots,m_{n})$, then $\supp
(f(w))=\{1,\ldots, i-1, i+1,\ldots, n\}$ and $|\supp (f(w))|=n-1$.

We now show that the value of $|\supp (f((x_nq_i)^k))|$ does not
depend on $k$.
\begin{lemma}\label{lem11}
    For every  $k\geqslant 1$ and $i=0,\ldots, n-2$ we have $|\supp (f((x_nq_i)^k))|=n-i-1$.
\end{lemma}
\begin{proof}
We show by induction that for every $k\geqslant1$, if
$k=m(n-i)+r$, where $0\leqslant r \leqslant n-i-1$, then
$f((x_nq_i)^k)(m_1,\ldots, m_n)$ has the form
\begin{equation*}
(\underbrace{m_{i+r+1}+m,\ldots, m_{i+r+1}+m}_{i+1}, m_{i+r+2}+m, \ldots, m_n+m, m_{i+2}+(m+1),\ldots, m_{i+r+1}+(m+1)),
\end{equation*}
where for $r=n-i-1$ the above expression is interpreted as
\begin{equation*}
(\underbrace{m_{i+r+1}+m,\ldots, m_{i+r+1}+m}_{i+1}, m_{i+2}+(m+1),\ldots, m_{i+r+1}+(m+1)).
\end{equation*}
This immediately yields the assertion. \\
First, notice that
\begin{equation}\label{wz1}
f(x_nq_i)(m_1,\ldots, m_n)=(\underbrace{m_{i+2}, m_{i+2}, \ldots, m_{i+2}}_{i+1}, m_{i+3},\ldots, m_n, m_{i+2}+1).
\end{equation}
In particular, for $k=1$ we have $|\supp (f((x_nq_i)^k))|=n-i-1$.
Assume that the claim holds for $k-1$. Then, applying (\ref{wz1})
and the induction hypothesis we see that
$f((x_nq_i)^k)(m_1,\ldots, m_n)=
f(x_nq_i)(f((x_nq_i)^{k-1}(m_1,\ldots, m_n)))$ has the form
\begin{equation*}
(\underbrace{m_{i+r+2}+m,\ldots, m_{i+r+2}+m}_{i+1}, m_{i+r+3}+m, \ldots, m_n+m, m_{i+2}+(m+1),\ldots, m_{i+r+2}+(m+1))
\end{equation*}
if $r\neq n-i-1$ and
\begin{equation*}
(\underbrace{m_{i+2}+m+1,\ldots, m_{i+2}+m+1}_{i+1}, m_{i+3}+m+1, \ldots, m_n+m+1, m_{i+2}+(m+2)),
\end{equation*}
otherwise. This proves the inductive step.
\end{proof}
 For every $i=-1, \ldots, n-2$ consider the following set
\begin{equation*}
Q_i=\{w\in C_n: |\supp (f(w))|\leqslant n-i-2 \}.
\end{equation*}
Then $Q_i$ is an ideal in $C_n$ for $i<n-2$ because for every
$x,y, w\in C_n$ we have $\supp (f(xw))\subseteq \supp (f(w))$ and
$\supp (f(wy))\subseteq \supp (f(w))$.  Thus, we get the following
chain of ideals
\begin{equation*}\emptyset= Q_{n-2}\s Q_{n-3}\s\cdots\s Q_{-1}\s C_n.
\end{equation*}
We will show that this chain is strongly related to the ideals
$I_j$ introduced in this section.
\begin{lemma}\label{lem12}
    For every $i=0,\ldots, n-2$ we have $Q_i\subseteq I_i$.
\end{lemma}
\begin{proof}
Suppose that for some $i\in\{0,\ldots, n-2\}$ there exists $w$
such that $w\in Q_i\setminus I_i$. Then, by the definition of the
ideals $I_i$ we get that for some $x,y\in C_n$ the element $xwy$
has the form $(x_nq_i)^k$ for some $k\geqslant 1$. Then
Lemma~\ref{lem11} implies that $|\supp (f(xwy))|=n-i-1$. On the
other hand, $xwy\in Q_i$, and thus $|\supp (f(xwy))|\leqslant
n-i-2$, which leads to a contradiction.
\end{proof}
 The following technical lemma will be used in the proof of Theorem~\ref{tw3}.
\begin{lemma}\label{lem13}
    For all $n-1\geqslant j>i+1\geqslant1$ we have
    \begin{enumerate}
        \item $x_nx_1\cdots x_{i}x_{n-1}\cdots x_{i+1}\in I_{i-1}$;
        \item $w=x_{j}\cdots x_{i+2}x_nx_1\cdots x_{i+1}x_{n-1}\cdots x_{j+1}\in I_i$;
    \end{enumerate}
    where for $j=n-1$ we put $w=x_{n-1}\cdots x_{i+2}x_nx_1\cdots x_{i+1}$.
\end{lemma}
\begin{proof}
For $i\neq 0$, the first part  is a direct consequence of
Lemma~\ref{lem11} and Lemma~\ref{lem12}. Moreover,
$x_nx_{n-1}\cdots x_1\in I_{-1}$ by the definition of $I_{-1}$.

 We
will prove the second part. A direct computation shows that for
every  $(m_1,\ldots, m_n)\in\mathbb{Z}^n$ and for $j<n-1$ the
element $f(x_{j}\cdots x_{i+2}x_nx_1\cdots x_{i+1}x_{n-1}\cdots
x_{j+1})(m_1,\ldots, m_n)$ has the form
\begin{equation*}
(\underbrace{m_{i+2},\ldots, m_{i+2}}_{i+1}, m_{i+3}, \ldots, m_{j}, m_{j+2},m_{j+2},\ldots, m_{n}, m_{i+2}+1).
\end{equation*}
Similarly, if $j=n-1$ then it is easy to see that $f(w)(m_1,\ldots
m_n)$ has the form
\begin{equation*}
(\underbrace{m_{i+2},\ldots, m_{i+2}}_{i+1}, m_{i+3}, \ldots,m_{n-1}, m_{i+2}+1, m_{i+2}+1).
\end{equation*}
In both cases $|\supp (f(w))|= n-i-2$, so Lemma~\ref{lem12}
implies that $w\in I_i$.
\end{proof}

The following is a direct consequence of the definition of the
ideals $I_i$ and of Lemma~\ref{lem13}.

\begin{corollary}\label{wn2}
For every $i\in\{0,\ldots, n-2\}$ we have $(x_nq_i)^k\in
I_{i-1}\setminus I_i$. Moreover, for all $i\in\{1, \ldots, n-2\}$
and $j\in \{i+1, \ldots, n-1\}$ we have $x_{j}\cdots
x_{i+1}x_nx_1\cdots x_{i}x_{n-1}\cdots x_{j+1}\in I_{i-1}\setminus
I_i$.
\end{corollary}

 Next, we prove a result implying that, after factoring the ideal $I_i$, the set $\tilde{M_i}\cup\{\theta\}$
becomes a  semigroup.
This will be crucial for the results of the next section.
\begin{theorem}\label{tw3}
    Let $w\in C_n$. Then for every  $i\in\{0,\ldots, n-2\}$ we have $(x_nq_i)w(x_nq_i)\in\{(x_nq_i)^k: k\geqslant 2\}\cup I_i$.
\end{theorem}
\begin{proof}
We proceed by induction on the length of the (reduced) word $w$.
If $w$ is the empty word, the assertion is clear. Assume that
$|w|>0$.

 Let $\pref_1(w)=x_j$, so $w=x_jw'$. If  $j\leqslant i$
then $(x_nq_i)w(x_nq_i)\arr{(4j)}(x_nq_i)w'(x_nq_i)$ and the
assertion follows by induction. Hence, assume next that $n>j>i$.
Then
\begin{align*}
(x_nq_i)w(x_nq_i)&= x_nx_1\cdots x_ix_{n-1}\cdots
x_{i+1}x_jw'x(x_nq_i)\\
 &\arr{(5j)} x_nx_1\cdots x_ix_{n-1}\cdots
x_{j+1}x_{j-1}\cdots x_{i+1}x_jw'(x_nq_i)\\
&\arr{(2)} x_{j-1}\cdots x_{i+2}x_nx_1\cdots x_{i+1} x_{n-1}\cdots
x_jw' (x_nq_i).\end{align*} Lemma~\ref{lem13} implies that
$x_{j-1}\cdots x_{i+2}x_nx_1\cdots x_{i+1} x_{n-1}\cdots x_jw'
(x_nq_i)\in I_i$, so the assertion follows.

 Hence, for the rest of the proof we may assume that
$\pref_1(w)=x_n$. Then  $w=x_nu_1x_nu_2\cdots x_nu_k$ for some
$k\geqslant 1$ and some $u_j$ such that $|u_j|_n=0$.

 If $u_k$ is
the empty word, then  $w(x_nq_i)\arr{(1)}x_nu_1\cdots
x_nu_{k-1}(x_nq_i)$, and the assertion follows by the induction
hypothesis. Hence, assume that $|u_k|\geqslant 1$ and the word
$x_nu_k$ is in the reduced form. If $\pref_1(u_k)=x_j$
($u_k=x_ju_k'$), then $j\in\{1, n-1\}$, since otherwise
$x_nu_k\arr{(2)} x_jx_nu_k'$, which contradicts the assumption
that $w$ is in the reduced form. Therefore, we consider two cases:
\begin{enumerate}
    \item $u_k=x_{n-1}u$, where $|u|_n=0$. Then, by Lemma~\ref{lem1} it follows that
    $u_k=x_{n-1}x_{n-2}\cdots x_j$ for some $n-1\geqslant j\geqslant1$.
    Then $x_nu_k(x_{n}q_i)= x_{n}x_{n-1}\cdots x_j(x_nq_i)$. If $j>1$ then
    $x_nu_k(x_{n}q_i)\arr{(5n)} u_k(x_nq_i)$ and $(x_nq_i)w(x_nq_i)$ after this reduction still contains at least two factors
    of the form $x_nq_i$; so the assertion follows by the induction hypothesis.
    Similarly, if $j=1$ and $i\neq 0$ then we get
    $x_nu_k(x_{n}q_i)\arr{(51)} x_nx_{n-1}\cdots x_2(x_nq_i)$ and then
    the word $x_{n-1}\cdots x_2$ is shorter than $u_k$. On the other hand, if
    $j=1$ and $i=0$, then $x_nu_k(x_{n}q_i)=(x_nq_i)^2$ and the assertion also follows from the induction hypothesis.
    \item $u_k=x_1u$, where $|u|_n=0$. Then by Lemma~\ref{lem2} we know that
    $x_nu_k=x_nx_1\cdots x_j$, where $1\leqslant j<n-1$ or
    $x_nu_k=x_nx_1\cdots x_rx_{n-1}\cdots x_s$ for some $1\leqslant r<s\leqslant n-1$.

    Let $x_nu_k=x_nx_1\cdots x_j$ and consider the case where $j\leqslant i$.
    Then  $x_nu_k(x_nq_i)\arr{(5j)} x_nx_1\cdots x_{j-1}(x_nq_i)$, the number of factors $x_nq_i$
    after this reduction is at least $2$ and the assertion follows from the induction hypothesis.
    If $n-1>j>i+1$, then
    \begin{align*}
    x_nu_k(x_nq_i)&\arr{(4)} x_nx_1\cdots x_jx_{n-1}\cdots x_{i+1}
    \arr{(4j)} x_nx_1\cdots x_jx_{n-1}\cdots x_{j+1}x_{j-1}\cdots x_{i+1}\\
    &\arr{(2)}x_nx_1\cdots x_jx_{j-1}\cdots x_{i+1}x_{n-1}\cdots x_{j+1}
    \arr{(4)}x_nx_1\cdots x_jx_{n-1}\cdots x_{j+1}\in I_i.
    \end{align*}
    Similarly, if $j=i+1$, then
    \begin{align*}
    x_nu_k(x_nq_i)&\arr{(4)} x_nx_1\cdots x_jx_{n-1}\cdots x_{i+1}
    \arr{(4j)} x_nx_1\cdots x_{j}x_{n-1}\cdots x_{j+1}\in I_i
    \end{align*}
    and the assertion also follows.

     Let $x_nu_k=x_nx_1\cdots x_rx_{n-1}\cdots x_s$, where $1\leqslant r<s\leqslant n-1$.
    \begin{itemize}
        \item If $s\leqslant i$, then $x_nx_1\cdots x_rx_{n-1}\cdots x_s(x_nq_i)\arr{(5s)}
        x_nx_1\cdots x_rx_{n-1}\cdots x_{s+1}(x_nq_i)$,
        and the obtained word also has at least two factors $x_nq_i$; so the assertion follows by induction.
        \item If $s=i+1$, then $r<i+1$. If additionally $r<i$, then
        $x_nx_1\cdots x_rx_{n-1}\cdots x_s(x_nq_i)\arr{(5r)} x_nx_1\cdots x_{r-1}x_{n-1}\cdots x_s(x_nq_i)$,
        and the obtained word still contains at least two factors $x_nq_i$; so the assertion follows by induction.
        \\If $r=i$, then $x_nu_k(x_nq_i)=(x_nq_i)^2$ and we are done by induction.
        \item If $n-1>s>i+1$, then
        \begin{align*} x_nu_k(x_nq_i)& \arr{(4s)} x_nx_1\cdots x_rx_{n-1}\cdots x_{s}x_n x_1\cdots
        x_ix_{n-1}\cdots x_{s+1}x_{s-1}\cdots x_{i+1}\\
        & \arr{(2),(3)} x_nx_1\cdots x_rx_{n-1}\cdots x_{i+2}x_nx_1\cdots x_{i+1}x_{n-1}\cdots x_{s+1}.\end{align*}
        Similarly, if $n-1=s>i+1$, then
        \begin{align*} x_nu_k(x_nq_i)& \arr{(4s)} x_nx_1\cdots x_rx_{n-1}x_n x_1\cdots x_ix_{n-2}\cdots x_{i+1}\\
        & \arr{(2),(3)} x_nx_1\cdots x_rx_{n-1}\cdots x_{i+2}x_nx_1\cdots x_{i+1}.\end{align*}
        Lemma~\ref{lem13} implies that $ x_nx_1\cdots x_rx_{n-1}\cdots x_{i+2}x_nx_1\cdots x_{i+1}\in
        I_i$. Therefore we also get $x_nx_1\cdots x_rx_{n-1}\cdots x_{i+2}x_nx_1\cdots x_{i+1}x_{n-1}\cdots x_{s+1}\in I_i$,
        and $w(x_nq_i)\in I_i$.
    \end{itemize}
\end{enumerate}
    This completes the proof.
\end{proof}

\subsection{Structures of matrix type}\label{s: stm}
Our next aim is to refine the information on the ideal chain
 \begin{equation*}\emptyset =I_{n-2} \vartriangleleft  I_{n-3}\vartriangleleft \cdots \vartriangleleft I_0\vartriangleleft
  I_{-1}  \vartriangleleft C_n
    \end{equation*}
of $C_n$ defined in the previous section. We will show that every
factor $I_{j-1}/I_j$, for $j=0,\ldots, n-2$, is, up to finitely
many elements, a semigroup of matrix type over a cyclic semigroup
and also that $C_n/I_{-1}$ is finite. Namely, the elements of the
family  $\tilde{M}_j$, described in Definition~\ref{matrix
struct}, with a zero element adjoined, treated as elements of the
Rees factor $I_{j-1}/I_j$, form a semigroup of matrix type. Using
certain natural involutions on $C_n$, we will also show that the
corresponding sandwich matrices are square matrices and they are
symmetric. In particular, this means that, for every $j$, there is
a bijection between the sets $A_j$ and $B_j$, which is not clear
directly from the description obtained in Theorem~\ref{tw1}.

Recall the definition of the sets $\tilde{M_i}$ and $\tilde{M}$
(Definition~\ref{matrix struct}).  For every $i=0,\ldots, n-2$ we
write
\begin{equation*}
    \tilde{M_i}=\{a(x_nq_i)^kb\in C_n: a\in A_i, b\in B_i, k\geqslant 1\}.
\end{equation*}
This is the set of elements considered in Theorem~\ref{tw1}. In
what follows, we identify elements of $C_n$ with the corresponding
reduced words. Hence, $\tilde{M}=\bigcup_{i=0}^{n-2}\tilde{M_i}$
consists of elements of $C_n$ that have (in the reduced form) a
factor of the form $x_nq_i$, for some $i$. Moreover, from
Proposition~\ref{stw1} we know that almost all elements of $C_n$
are in this set.

Certain involutions that preserve the ideals $I_{i-1}$ and sets
$\tilde{M}_{i}$, for $i\in\{0, \ldots, n-2\}$, will be useful in
this context. In particular, they can be used to establish an
internal symmetry of every set $\tilde{M}_{i}$.

\begin{defi} \label{sigma-def}
    Let $\tau: C_n \longrightarrow C_n$
    be the involution defined in Definition~\ref{involut}.
    So $\tau(x_i)=x_{n-i} $ for $i\neq n$ and $\tau (x_n) = x_n$.
    Let $\sigma: \langle x_1,\ldots, x_n\rangle\longrightarrow \langle x_1,\ldots, x_n\rangle$
    be the automorphism such that  $\sigma(x_i)=x_{i+1}$ for every $i=1,\ldots, n$, where we put
    $x_{n+1}=x_1$.  It is easy to check that
    $\sigma$ preserves the set of defining relations of $C_n$. Hence, $\sigma$ can be viewed as an automorphism of $C_n$.
    Therefore, the map $\sigma^{i}\tau$ also is an involution of $C_n$, for $i=0,\ldots, n-1$.
\end{defi}

\begin{corollary}\label{wn4}
    The involution $\chi_{i} =\sigma^{i+1}\tau$ satisfies: $\chi_i ((x_nq_i)^k)=(x_nq_i)^k$ for every $k\geqslant
    1$, $\chi_i (A_i)=B_i, \chi_i (B_i)=A_i$, and in particular $\chi_i (\tilde{M}_{i}) =
    \tilde{M}_i$.
\end{corollary}
\begin{proof}
Notice that $\chi_i (x_nx_1\cdots x_{i})=x_1\cdots x_{i+1}$ and
$\chi_i (x_{n-1}\cdots x_{i+1})=x_nx_{n-1}\cdots x_{i+2}$. Hence
$\chi_i (x_nq_i) = (x_nx_{n-1}\cdots x_{i+2} )(x_1\cdots
x_{i+1})=x_nq_i$ in $C_n$, so for $k\geqslant 1$ we get $\chi_i
((x_nq_{i})^k)=(x_nq_{i})^k$ for every $i$. Let $w=(x_nq_i)^m b\in
\tilde{M}_{i}$, where $b\in B_i$ and $m\geqslant 1$. Then $\chi_i
(w)= \chi_i (b) (x_nq_i)^m$. Moreover, the reduced form of $\chi_i
(b) (x_nq_i)^m$ is obtained by moving certain generators to the
left (other reductions are not possible because they reduce the
length of the word, while $\chi_i^2 =id$). But $x_n$ cannot be
moved to the left, so the consecutive generators in $(x_nq_i)^m$
cannot be moved to the left either. It follows that the reduced
form of $\chi_i (w)$ is equal to $a (x_nq_i)^m$, where $a$ is the
reduced form of $\chi_i (b)$. If $a=u(x_nq_i)^r$ for some
$r\geqslant 1$ and some $u$, then $b=\chi_i(a) =(x_nq_i)^{r}\chi_i
(u)$ and $w= \chi_i (a(x_nq_i)^m)= (x_nq_i)^{r+m}\chi_i (u)$. And
again, since no generator can be moved into the word $x_nq_i$ from
the right (without making the word shorter), it follows that we
cannot have $r>0$.  Thus, $a\in A_{i}$. Therefore,
$\chi_i(B_{i})\subseteq A_{i}$. A symmetric argument shows that
$\chi_i(A_{i})\subseteq B_{i}$. Since $\chi_i^2 = id$, we get
$\chi_i(B_{i})= A_{i}$ and $\chi_i(A_{i})= B_{i}$. The result
follows.
\end{proof}

As noticed in Lemma~\ref{lem4}, if $i=0$ or $i=n-2$, then reduced
words in $C_n$ that have a factor of the form $x_nq_i$ must come
from the infinite word $(x_nq_i)^{\infty}$. It is then clear that
for such a word $s$ we can find $w$,$z\in C_n$ such that
$wsz\in\langle x_nq_i\rangle$. We will show  that the latter
property remains valid for all $i$.
\begin{theorem}\label{tw2}
    Let $i\in\{0, \ldots, n-2\}$. Then:
    \begin{enumerate}
        \item[1)] for every $a\in A_i$ there exists $w\in C_n$ such that $wa\in \langle x_nq_i\rangle$;
        \item[2)] for every $b\in B_i$ there exists $w\in C_n$ such that $bw\in\langle x_nq_i\rangle$.
    \end{enumerate}
\end{theorem}

\begin{proof} It is clear that the assertion holds for $i=0, n-2$. So, we will
further assume that $i\in\{1,\ldots, n-3\}$. In view of
Corollary~\ref{wn4} it is enough to prove assertion 2). We will
use notation of Theorem~\ref{tw1}. This includes the notion of
blocks, introduced directly after the formulation of this theorem.

     The description of the set $B_i$  implies directly that every $b\in B_i$
    is a prefix of the following element
    \begin{equation*}
    b'= x_n(x_{1}\cdots x_{i_1}x_{n-1}\cdots x_{j_1})\cdots x_n(x_{1}\cdots x_{i_r}x_{n-1}\cdots x_{j_r})x_nx_{n-1}\cdots x_{j_{r+1}}.
    \end{equation*}
    Clearly, if the assertion holds for some word then it also holds for
    every prefix of this word. Thus, it is enough to prove the assertion in the case
    where $b=b'$. We proceed by induction on the number of blocks in the word $b$ of
    this type; so on $r+1$. Let $r=0$. Then $b=x_nx_{n-1}\cdots x_{j_1}$ for some $j_1>i+1$.
    If we define $w=x_1\cdots x_ix_{j_1-1}\cdots x_{i+1}$ then
    $bw\arr{(2)}x_nx_1\cdots x_ix_{n-1}\cdots x_{i+1}$, as desired.

    So, assume that the assertion holds for all words in $B_i$,
    that are of the form

    $$x_n(x_{1}\cdots x_{i_1}x_{n-1}\cdots x_{j_1})\cdots
    x_n(x_{1}\cdots x_{i_k}x_{n-1}\cdots x_{j_k})x_nx_{n-1}\cdots
    x_{j_{k+1}},$$
 where $k<r$ (the number of blocks is $k+1$). Let $b$ be the following word with $r+1$
    blocks:
    \begin{equation*}b=x_n(x_{1}\cdots x_{i_1}x_{n-1}\cdots x_{j_1})\cdots x_n(x_{1}
    \cdots x_{i_r}x_{n-1}\cdots x_{j_r})x_nx_{n-1}\cdots x_{j_{r+1}}.\end{equation*}
    Let  $w'=x_{i_1+1}\cdots x_ix_{j_1-1}\cdots x_{i+1}$, where for $i_1=i$,
    we put $w'=x_{j_1-1}\cdots x_{i+1}$.
    Since for $k=r, r+1$ we have inequalities  $1<j_{k}-i<n-1$, using reduction
    (2) applied several times, it is easy to see that the word $bw'$ can be reduced to
    \begin{align*}
    x_n(x_{1}\cdots x_{i_1}x_{n-1}\cdots x_{j_1})\cdots x_n(x_{1}\cdots x_{i_r}(x_{i_1+1}
    \cdots x_i)x_{n-1}&\cdots x_{j_r})x_nx_{n-1}\cdots x_{j_{r+1}}(x_{j_1-1}\cdots x_{i+1})
    \end{align*}
    Moreover, we know that  $i_r+1<i_1+1<n-1$, so applying reduction  (3)
    several times, we can rewrite the word $bw'$ to the form
   \begin{align*}
   x_n(x_{1}\cdots x_{i_1}x_{n-1}\cdots x_{j_1})\cdots (x_{i_1+1}\cdots x_i)x_n(x_{1}
   \cdots x_{i_r}x_{n-1}&\cdots x_{j_r})x_nx_{n-1}\cdots x_{j_{r+1}}(x_{j_1-1}\cdots x_{i+1}).
   \end{align*}
   Repeating this process, alternately using the inequalities
   $1<j_k-i<n-1$ for $k>1$ (and applying reduction (2)), and using the inequalities
 $i_k+1<i_1+1<n-1$  for $k>1$ (and applying reduction (3)), we
 come to the following form of the element $bw'$:
   \begin{equation*}
   x_n(x_1\cdots x_ix_{n-1}\cdots x_{j_1})\cdots x_n(x_{1}\cdots x_{i_r}x_{n-1}
   \cdots x_{j_r})x_nx_{n-1}\cdots x_{j_{r+1}}(x_{j_1-1}\cdots x_{i+1}).
   \end{equation*}
  Similarly, applying alternately inequalities  $1<m-(j_{1}-1)<n-1$ for $m=j_{2},\ldots ,n$
  (and reduction  (2)) and inequalities  $i_k+1<m+1<n-1$ for $m=i+1, \ldots, j_1-1$, $k=2,\ldots ,r$
  (and reduction (3) applied several times) we can finally rewrite the word $bw'$ to the form
  \begin{align*}
  x_n(x_1\cdots x_ix_{n-1}\cdots x_{i+1})x_n(x_{1}\cdots x_{i_2}x_{n-1}\cdots x_{j_2})
  \cdots x_n(x_1\cdots x_{i_r}x_{n-1}&\cdots x_{j_r})x_nx_{n-1}\cdots x_{j_{r+1}}.
  \end{align*}
    Then $b'=x_n(x_{1}\cdots x_{i_2}x_{n-1}\cdots x_{j_2})\cdots x_n(x_1\cdots x_{i_r}x_{n-1}
    \cdots x_{j_r})x_nx_{n-1}\cdots x_{j_{r+1}} \in B_i$.
    Thus, by the inductive hypothesis, there exists $v\in C_n$ such that
     $b'v\in \langle x_nq_i\rangle$. Hence, for $w=w'v$ we have $bw\in \langle x_nq_i\rangle$.
This completes the proof.
\end{proof}

Next we show that for every $i$ the set $\tilde{M_i}$ is contained
in $I_{i-1}\setminus I_i$.

\begin{stw}\label{stw2}
    For every $i\in\{0,\ldots, n-2\}$ we have $\tilde{M_i}\subseteq I_{i-1}\setminus I_i$.
\end{stw}
\begin{proof}
Since $I_{i}$ is an ideal of $C_n$, Lemma~\ref{lem13} implies that
$\tilde{M_i}\subseteq I_{i-1}$. Let $a(x_nq_i)^kb\in\tilde{M_i}$.
Theorem~\ref{tw2} implies that there exist $w, v\in C_n$ such that
$wa, bv\in\langle x_nq_i\rangle$. So, $wa(x_nq_i)^kbv\in\langle
x_nq_i\rangle$ and $a(x_nq_i)^kb\notin I_{i}$. The assertion
follows.
\end{proof}

\begin{corollary}\label{wn6}
    For every $i\in\{0,\ldots, n-2\}$ the set $(I_{i-1}\setminus I_i)\setminus\tilde{M_i}$ is finite.
\end{corollary}
\begin{proof}
By Proposition~\ref{stw2}, $\tilde{M_i}\s I_{i-1}\setminus I_i$.
Let $j>i$. Then $I_{j-1}\s I_i$, so that $(I_{i-1}\setminus
I_i)\cap(I_{j-1}\setminus I_j)=\emptyset$. In particular, for
$j\neq i$ we get $\tilde{M_j}\cap (I_{i-1}\setminus
I_i)=\emptyset$. Since from Proposition~\ref{stw1} we know that
$C_n\setminus\tilde{M}$ is finite, it follows that
$(I_{i-1}\setminus I_i)\setminus\tilde{M_i}$ is also finite.
\end{proof}

By Proposition~\ref{stw1} we know that $C_n\setminus\tilde{M}$ is
a finite set. Moreover, Proposition~\ref{stw2} implies that for
every $i=0,\ldots, n-2$ we have $\tilde{M_i}\s I_{i-1}\setminus
I_i\s I_{-1}$, so that also $\tilde{M}\s I_{-1}$. Our next
observation follows.

\begin{corollary}\label{I-1}
$C_n/I_{-1}$ is a finite semigroup.
\end{corollary}

The following two simple lemmas will be also useful.

\begin{lemma}\label{sigma}
For every $i\in\{0, \ldots, n-3\}$ we have $\sigma(I_i)=I_i$.
Moreover, $\sigma (w)\in \tilde{M}_i$ for almost all $w\in
\tilde{M}_i$, if $i\in \{ 0,\ldots , n-2\}$.
\end{lemma}
 \begin{proof}
 Clearly, $\sigma^n=id$ and $\sigma: C_n\longrightarrow C_n$ is an automorphism.
 Hence, for the first assertion, it is enough to prove that
  $\sigma(C_n\setminus I_i)\s C_n\setminus I_i$. Let $w\in C_n\setminus I_i$. Then there exist $u, v\in C_n$ such that
   $uwv=(x_nq_i)^m$, where $m\geqslant 1$.
 A direct computation shows that
 $\sigma((x_nq_i)^m)=(x_1x_2\cdots x_{i+1}x_nx_{n-1}\cdots x_{i+2})^m\arr{(2)}(x_1\cdots x_{i+1})(x_nq_i)^{m-1}(x_n\cdots x_{i+2})$.
 In particular, $x_1\cdots x_{i+1}\in A_i$, $x_n\cdots x_{i+2}\in B_i$,
 so from Theorem~\ref{tw2} we know that there exist $u', v'\in C_n$ such that
 $u'(x_1\cdots x_{i+1})(x_nq_i)^{m-1}(x_n\cdots x_{i+2})v'\in\langle
 x_nq_i\rangle$.
 Then for $p=u'\sigma(u)$ and $q=\sigma(v)v'$ we get $p\sigma(w)q\in\langle x_nq_i\rangle$,
 so that $\sigma(w)\in C_n\setminus I_i$, as desired.

 The second assertion follows now directly from
 Corollary~\ref{wn6} (if $i>0$) and from Corollary~\ref{I-1} (if
 $i=0$).
\end{proof}

\begin{lemma}\label{lem14}
    For every $i\in\{1,\ldots, n-2\}$ and every nonnegative integer
$m$ we have $\sigma^{m}\tau (I_{i-1}\setminus
I_i)=I_{i-1}\setminus I_i$ and $\sigma^{m}\tau (I_{i-1})=I_{i-1}$.
\end{lemma}
\begin{proof}
We have seen in the proof of Corollary~\ref{wn4} that
$\sigma^{i}\tau((x_nq_{i-1})^k)=(x_nq_{i-1})^k$ for $k\geqslant
1$.

Let $w\in C_n$ be such that $\sigma^{i}\tau(w)\notin I_{i-1}$.
Then, by the definition of $I_{i-1}$  there exist $x,y\in C_n$
such that $x\sigma^{i}\tau(w)y=(x_nq_{i-1})^j$, for some $j$. Then
$\sigma^{i}\tau(y)w\sigma^{i}\tau(x)=(x_nq_{i-1})^j$, so in
particular $w\notin I_{i-1}$. This shows that
$\sigma^{i}\tau(I_{i-1})\subseteq I_{i-1}$.

Since $\sigma^n = id$, from Lemma~\ref{sigma} it now follows that
$\tau (I_{i-1})\subseteq I_{i-1}$, and thus $\tau (I_{i-1}) =
I_{i-1}$ because $\tau$ is an involution. Therefore, $\sigma^m
\tau (I_{i-1})=I_{i-1}$ for every nonnegative $m$. The assertion
follows.
\end{proof}

In the two extreme cases, namely for $i=0$ and $i=n-2$, the
description of $M_{i}$ is quite simple (see Lemma~\ref{lem4}). In
particular, $\tilde{M}_{n-2}$ coincides with the set of all
factors of the word $(x_nx_1\cdots x_{n-1})^{\infty}$, that
contain a factor $x_nx_1\cdots x_{n-1}$. Moreover, our earlier
results lead to the following consequence.

\begin{corollary}\label{M_{n-2}}
    $\tilde{M}_{n-2}=C_n(x_nx_1\cdots x_{n-1})C_n$.
\end{corollary}
\begin{proof} The definition implies that $\tilde{M}_{n-2}\s C_n(x_nx_1\cdots x_{n-1})C_n$.
By Lemmas~\ref{lem o mno} and \ref{lem o mno dual} we know that
every element of $C_n(x_nx_1\cdots x_{n-1})C_n$ has reduced form
that is a factor of $(x_nx_1\cdots x_{n-1})^{\infty}$. Moreover,
the second parts of both lemmas imply that if $w(x_nx_1\cdots
x_{n-1})v\in C_n(x_nx_1\cdots x_{n-1})C_n$, then the reduced form
of this word has a factor $x_nx_1\cdots x_{n-1}$.
 \end{proof}

In the second  extreme case, namely when $i=0$, we have
$\tilde{M}_0 \subseteq C_n(x_nq_{0})C_n\subseteq I_{-1}$.
Moreover, equality holds modulo the ideal $I_0$, as proved in the
following lemma.

\begin{lemma}\label{lemI_{-1}}
    $I_{-1}=\tilde{M_0}\cup I_0$.
\end{lemma}
\begin{proof}
By the definition, $I_{-1}=I_0\cup C_n(x_nq_0)C_n$. It is clear
that $\tilde{M}_0\s C_n(x_nq_0)C_n$, so also $\tilde{M}_0\cup
I_0\s I_{-1}$. Since $x_nq_0\in \tilde{M}_0$, it is enough to
prove that $\tilde{M}_0\cup I_0$ is an ideal in $C_n$. Since
$\sigma \tau (\tilde{M}_0)=\tilde{M}_0$ and $\sigma \tau
(I_0)=I_0$ by Corollary~\ref{wn4} and Lemma~\ref{lem14}, it is
enough to check that $\tilde{M}_0\cup I_0$ is a right ideal of
$C_n$. So, for every $w\in\tilde{M}_0\cup I_0$ and
$x_j\in\{x_1,\ldots, x_n\}$ we have $wx_j\in\tilde{M}_0\cup I_0$.
If $w\in I_0$, then $wx_j\in I_0$, because  $I_0$ is an ideal of
$C_n$.

 So, let
$w=a(x_nx_{n-1}\cdots x_1)^mb\in\tilde{M_0}$, where $a\in A_0,
b\in B_0$. From Theorem~\ref{tw1} we know that
 $b=x_nx_{n-1}\cdots x_l$ for some $l\in\{2,
\ldots, n\}$ or $b$ is the empty word. First, notice that by
Lemma~\ref{lem13}, for $j\neq 1, n$, we have
\begin{align*}
(x_nx_{n-1}\cdots x_1)x_j\arr{(5j)}x_n\cdots x_{j+1}x_{j-1}\cdots
x_1x_j\arr{(2)}x_{j-1}\cdots x_2x_nx_1x_{n-1}\cdots x_j \in I_0 .
\end{align*}
Clearly, for $j=1, n$ we have $(x_nx_{n-1}\cdots x_1)x_j
\in\tilde{M_0}$. This implies that the assertion holds if $b$ is
the empty word.

Hence, assume that $b=x_nx_{n-1}\cdots x_l$ for some $l$. We will
show that $(x_nx_{n-1}\cdots x_1)x_nx_{n-1}\cdots x_lx_j\in I_0$
for every $j\neq l-1, l$. This will prove the result because $I_0$
is an ideal in $C_n$. Notice that for $j\neq l-1, l$ we have
\begin{align*}
(x_nq_0)&x_nx_{n-1}\cdots x_lx_j\\&\arr{(*)}
\begin{cases}(x_nq_0)(x_n\cdots x_{j+1})(x_{j-1}\cdots x_1)x_j\arr{(2)}(x_n\cdots x_1)x_{j-1}\cdots x_1(x_n\cdots x_{j})\textrm{ if }j>l\\
(x_nq_0)x_j(x_n\cdots x_l)\textrm{ if }1<j<l-1\\(x_{n}\cdots
x_2)(x_n\cdots x_l)x_1\arr{(2), (5n)}x_{n-1}\cdots
x_2x_nx_1x_{n-1}\cdots x_l \textrm{ if }j=1 , \end{cases}
\end{align*}
where $\arr{(*)}$ denotes $\arr{(5j)}$, $\arr{(2)}$ and $\arr{(5j)}$, respectively.\\
In each case,  $(x_nq_0)x_nx_{n-1}\cdots x_lx_j\in I_0$. In the
first case, this is a consequence of the fact that  $I_0$ is an
ideal in  $C_n$ and $j-1\neq1, n$, so the previous computation
implies that $(x_nq_0)x_j\in I_0$. Similarly, in the second case
$1<j<l-1\leqslant n-1$, so that also $(x_nq_i)x_j\in I_0$. In the
third case, from Lemma~\ref{lem13} and since $I_0$ is an ideal we
know that $x_{n-1}\cdots x_2x_nx_1x_{n-1}\cdots x_l\in I_0$. The
result follows.
\end{proof}

We are now in a position to improve the assertion of
Theorem~\ref{tw2}.
\begin{corollary} \label{multip-rows}
    Let $i\in\{0, \ldots, n-2\}$. Then
    \begin{enumerate}
        \item for every $a\in A_i$ there exists $w\in\tilde{M_i}$ such that $wa\in \langle
        x_nq_i\rangle$,
        \item for every $b\in B_i$ there exists $v\in\tilde{M_i}$ such that $bv\in\langle x_nq_i\rangle$.
    \end{enumerate}
    Consequently, $x\mapsto wx$ is an injective map $a \langle x_nq_i\rangle B_i \longrightarrow
     \langle x_nq_i\rangle B_i$, and
     $x\mapsto xv$ is an injective map $ A_i \langle x_nq_i\rangle b  \longrightarrow
     A_i\langle x_nq_i\rangle $.
\end{corollary}
\begin{proof}
We only prove the first part;  the second part then follows from
Corollary~\ref{wn4}. Let $a\in A_i$. From Theorem~\ref{tw2} we
know that there exists $u\in C_n$ such that $ua=(x_nq_i)^k$ for
some $k\geqslant 1$. The elements $(x_nq_i)^mu\in C_n$, where
$m\geqslant 1$ are pairwise different. Indeed, suppose that for
some $m_1, m_2\geqslant 1$ in $C_n$ we have
$(x_nq_i)^{m_1}u=(x_nq_i)^{m_2}u$. Then also
$(x_nq_i)^{m_1}ua=(x_nq_i)^{m_2}ua$, so that
$(x_nq_i)^{m_1+k}=(x_nq_i)^{m_2+k}$, and consequently $m_1=m_2$.
By Proposition~\ref{stw1} there exists $m\geqslant1$ such that
$(x_nq_i)^mu\in\tilde{M}$. According to Lemma~\ref{lem13}, for
every $m\geqslant1$ we have $(x_nq_i)^m\in I_{i-1}\setminus I_i$.
Since $I_{i-1}$ is an ideal of $C_n$, we get $(x_nq_i)^mu\in
I_{i-1}$. If $(x_nq_i)^mu\in I_{i}$, then also
$(x_nq_i)^{m}ua=(x_nq_i)^{m+k}\in I_{i}$, a contradiction.
Therefore, by Proposition~\ref{stw2},
$w=(x_nq_i)^mu\in\tilde{M}\cap (I_{i-1}\setminus
I_i)=\tilde{M}_i$, and the assertion follows.
\end{proof}

Notice that $\tilde{M_i}\subseteq C_n$, for $i< n-2$, is not
closed under multiplication. For example, consider
 $u=(x_nq_i)x_nx_{n-1}$, $w=x_nq_i$ for some $1\leqslant i\leqslant n-3$.
 It is easy to check that the reduced form of $uw$ is equal to
 $x_{n-2}\cdots x_{i+2}x_nx_1\cdots x_{i+1}x_{n-1}\cdots x_{i+2}x_nx_1\cdots x_{i+1}$,
so that indeed $uw\notin \tilde{M_i}$. However, from
Theorem~\ref{tw3} we know that if $u=a(x_nq_i)^kb,
w=a'(x_nq_i)^{k'}b'\in \tilde{M_i}$, then either
$uw=a(x_nq_i)^mb'\in \tilde{M}_{i}$ for some $m\geqslant 2$ or
$uw\in I_i$.

We will define semigroups $M_i$, with properties described in the
beginning of this section.
\begin{defi} \label{matrix}
    Let $i\in\{0,\ldots, n-2\}$. Consider the set $M_i=\tilde{M_i}\cup\{\theta\}$
    with operation, defined for any  $u=a(x_nq_i)^kb, w=a'(x_nq_i)^{k'}b'\in\tilde{M_i}$ by
    \begin{equation*}uw=\begin{cases}a(x_nq_i)^kba'(x_nq_i)^{k'}b' & \textrm{ if }(x_nq_i)ba'(x_nq_i)\in \langle x_n q_i\rangle \\
    \theta &\textrm{ if }(x_nq_i)ba'(x_nq_i)\in I_i\end{cases}
    \end{equation*}
    and  $w\theta=\theta w=\theta$ for every $w\in M_i$. Then, by Theorem~\ref{tw3}
    the definition is correct and $M_i$  is a semigroup under this operation.
    \end{defi}
\noindent These semigroups can be interpreted as Rees factor
semigroups. Namely, for $i\leqslant n-3$, $I_i$ is an ideal of
$C_n$, and we may consider the factor semigroup $C_n/I_i$. In
other words, $C_n/I_i$ is the semigroup $(C_n\setminus I_i)\cup
\{\theta\}$ with zero $\theta$ and with operation
\begin{equation*}
s\cdot t=\begin{cases}st & \textrm{ if }st\notin I_i \\\theta
 & \textrm{ if }st\in I_i . \end{cases}
\end{equation*}
 While $I_{n-2}=\emptyset$, for every subsemigroup $J$ of
$C_n$ we define $J/I_{n-2}=J^0$; the semigroup $J$ with zero
adjoined. Notice that $J_{i}=\tilde M_{i}\cup I_{i}$ is a
subsemigroup of $I_{i-1}$ by Theorem~\ref{tw3} and
Proposition~\ref{stw2}. Thus, our definition yields
$M_i=J_{i-1}/I_i\s C_n/I_i$.

From Lemma~\ref{lemI_{-1}} we know that $I_{-1}/I_0=M_0$.
Moreover, Lemma~\ref{lem o mno} implies that
$\tilde{M}_{n-2}=M_{n-2}\setminus \{\theta\}$ is an ideal in
$C_n$.

We will show that all $M_i$ defined above are semigroups of matrix
type, see \cite{semalg}, Section~5. Let $S_i$  denote the cyclic
semigroup generated by $x_nq_i$.
\begin{corollary}\label{wn3}
    $M_i$ is a semigroup of matrix type. Namely, $M_i\cong \mathcal{M}^0(S_i, A_i, B_i; P_i)$,
    where $P_i$ is a matrix of size $B_i\times A_i$ with coefficients in
    $\langle x_nq_i\rangle\cup\{\theta\}$.
\end{corollary}
\begin{proof}
Let $P_i=(p_{ba})$ be the matrix defined as follows
\begin{equation}  \label{sandwich}
p_{ba}=\begin{cases}(x_nq_i)^{\alpha-2} & \textrm{ if }
(x_nq_i)ba(x_nq_i)= (x_n q_i)^{\alpha}\in\langle x_nq_i\rangle
\\\theta & \textrm{ if }(x_nq_i)ba(x_nq_i)\in I_i . \end{cases}
\end{equation}
We define  $\phi:M_i\longrightarrow \mathcal{M}^{0}(S_i, A_i, B_i;
P_i)$ by the formula $\phi(a(x_nq_i)^kb)=((x_nq_i)^{k}; a, b)$ and
$\phi(\theta)=\theta$. From the uniqueness of the reduced forms of
elements of $M_i$ and by Proposition~\ref{stw2} this function is
well defined. It is clear that $\phi$ is bijective. Moreover, we
claim that $\phi$ is a homomorphism. Indeed, let $u=a(x_nq_i)^kb,
w=a'(x_nq_i)^{k'}b'\in M_i$.

If $uw\notin I_i$, then $(x_nq_i)ba'(x_nq_i)=(x_nq_i)^{\alpha}$
    for some $\alpha\geqslant 2$ and then
    \begin{align*}
    \phi(u)\phi(w)&=((x_nq_i)^{k}; a, b)((x_nq_i)^{k'}; a', b')
    \\&=((x_nq_i)^{k}p_{ba'}(x_nq_i)^{k'}; a, b')=((x_nq_i)^{k+k'+\alpha-2}; a, b')
    \\&=\phi(a(x_nq_i)^{k-1+k'-1+\alpha}b')
    \\&=    \phi (a(x_nq_i)^{k}ba'(x_nq_i)^{k'}b')= \phi(uw).
\end{align*}
If $uw\in I_i$ then it is easy to see that
$\phi(uw)=\theta=\phi(u)\phi(w)$, because $p_{ba'}=\theta$. This
completes the proof.
 \end{proof}

\begin{remark} \label{symmetric}
Assume that $x_nq_i ba x_nq_i = (x_nq_i)^{\alpha}$ for some $a\in
A_i, b\in B_i$.  Then
\begin{eqnarray*}
(x_nq_i)^{\alpha}=\chi_i((x_nq_i)^{\alpha})= \chi_i (x_nq_i)
\chi_i (ba)\chi_i (x_nq_i) = x_nq_i\chi_i (a) \chi_i (b) x_nq_i .
\end{eqnarray*} By Corollary~\ref{wn4}, $\chi_i$ determines a
bijection between the sets $A_i$ and $B_i$.  Hence, from
(\ref{sandwich}) in the proof of Corollary~\ref{wn3} it follows
that the matrix $P_{i}$ is symmetric, if the ordering of the
elements of the set $A_i$ corresponds to the ordering of their
images under $\chi_i$ (see the examples in
Section~\ref{examples}).
\end{remark}

The main results of this section can be now summarized as follows.
\begin{corollary}\label{struktura}
     $C_n$ has a chain of ideals
    \begin{equation*}\emptyset =I_{n-2} \vartriangleleft  I_{n-3}\vartriangleleft \cdots \vartriangleleft I_0\vartriangleleft
  I_{-1}  \vartriangleleft C_n
    \end{equation*}
    with the following properties
    \begin{enumerate}
        \item for $i=0,\ldots, n-2$ there exist semigroups of matrix type
        $M_i=\mathcal{M}^0(S_i, A_i, B_i; P_i)$,
        such that $M_i\s I_{i-1}/I_i$, where $S_{i}$ is the cyclic semigroup generated
        by $x_nq_i$, $P_i$
        is a square symmetric matrix of size $B_i\times A_i$ and with coefficients in
        $\langle x_nq_i\rangle\cup\{\theta\}$;
        \item for $i=1,\ldots, n-2$ the sets $(I_{i-1}/I_i)\setminus M_i$ are finite;
        \item $I_{-1}/I_0=M_0$;
        \item $\tilde{M}_{n-2}=M_{n-2}\setminus \{\theta\}\vartriangleleft C_n$;
        \item  $C_n/I_{-1}$ is a finite semigroup.
    \end{enumerate}
\end{corollary}

\section{Examples}\label{examples}

In this section we illustrate our results with the simplest cases
of Hecke--Kiselman monoids $C_n$ associated to cyclic graphs, for
$n=3$ and $4$.

\subsection{Cycle of length $3$}
For simplicity, write $x_1=a$, $x_2=b$, $x_3=c$. Recall that
\begin{equation*}
C_3=\langle a, b, c: a^2=a, b^2=b, c^2=c, ab=aba=bab, bc=bcb=cbc,
ca=cac=aca \rangle.
\end{equation*}
The following is a  direct consequence of
Corollary~\ref{podstawowy wniosek}.

\begin{lemma}\label{lem n=3}
The reduced form of every element of $C_3$ is a factor of one of
the following infinite words: $(cab)^{\infty}$, $(cba)^{\infty}$.
\end{lemma}
Taking $n=3$ in Theorem~\ref{tw1}, we get that $\tilde{M_0}$
consists of all factors of $(cba)^{\infty}$, that have $cba$ as a
subfactor. Similarly, $\tilde{M_1}$ consists of all factors of
$(cab)^{\infty}$ that have a subfactor $cab$.

According to Corollary~\ref{struktura}, $C_3$ has an ideal chain
\begin{equation*}
I_0\subseteq I_{-1},
\end{equation*}
where $I_0=\{w\in C_3: C_3wC_3\cap \langle cba\rangle=\emptyset\}$
and $I_{-1}=I_0\cup\tilde{M_0}$.

\begin{lemma}\label{lem16}
     Let $T$ be the cyclic semigroup generated by $t=cab$. Then $M_1=(C_3cabC_3)^0$ is
     a semigroup of matrix type $\mathcal{M}^{0}(T, A_1, B_1; P_1)$, where
     $A_1=\{1, b, ab\}$, $B_1=\{1, c, ca\}$, with sandwich matrix
     (with coefficients in $T^1$)
    \begin{equation*}
    P_1=\begin{pmatrix}
    1 & 1 & 1\\1 & 1 & t\\ 1 & t & t
    \end{pmatrix}.
    \end{equation*}
\end{lemma}

Similarly, for $I_{-1}$ we have
\begin{lemma}\label{lem17}
     Let $S$ be the cyclic semigroup generated by  $s=cba$. Then the semigroup
     $M_0=I_{-1}/I_0$ is a
     semigroup of matrix type $\mathcal{M}^{0}(S, A_0, B_0; P_0)$,  $A_0=\{1, a, ba\}$,
     $B_0=\{1, c, cb\}$, with sandwich matrix (with coefficients in $S^1\cup\{\theta\}$)
    \begin{equation*}
    P_0=\begin{pmatrix}
    1 & 1 & \theta\\1 & \theta & s\\ \theta & s & s
    \end{pmatrix}.
    \end{equation*}
\end{lemma}

 Recall that the rows of $P_i$ are indexed by the set $B_i$, and
columns by the set $A_i$. For simplicity, we identify the elements
of these sets with $1, 2, 3$, in the order in which these elements
were listed. For example, the $(3,2)$-entry of the sandwich matrix
$P_0$ corresponds to the pair $(cb, a)$.

The above two lemmas follow directly from Corollary~\ref{wn3}. To
indicate computations that are used to determine the coefficients
of the sandwich matrices, let us focus on $P_0$. For simplicity,
if $\alpha\in A_0$, $\beta\in B_0$, then we write
\begin{equation*}\overline{p_{\beta\alpha}}=\begin{cases}(cba)\beta\alpha(cba) & \textrm{ if }
\beta\alpha\in \langle s\rangle\\  \theta &\textrm{ if
}\beta\alpha\in I_0 , \end{cases}
\end{equation*}
that is if  $\overline{p_{\beta\alpha}}=s^k$, then
$p_{\beta\alpha}=s^{k-2}$ and if $\overline{p_{\beta
\alpha}}=\theta$, then also  $p_{\beta \alpha}=\theta$. Then, for
example
$\overline{p_{(cb)(ba)}}=(cba)cbba(cba)\arr{(1)}(cba)ca(cba)\arr{(2)}(cba)^3$.
So, $p_{(cb)(ba)} =s$.

We derive the following consequence for the algebras $K_0[M_0]$
and $K_0[M_1]$.
\begin{corollary}
    Algebras  $K_0[M_0]$ and $K_0[M_1]$ are of matrix type. Namely,
    $K_0[M_1]=\mathcal{M}(K[T], A_1, B_1; P_1)$ and
    $K_0[M_0]=\mathcal{M}(K[S], A_0, B_0; P_0)$,
    where $T$ and $S$ are the cyclic semigroups generated by $t=cab$ and $s=cba$,
    respectively.
\end{corollary}

It is easy to see that $\det P_1=-(t-1)^2\neq 0$ and $\det
P_0=-s(s+1)\neq 0$, whence $P_1$ and $P_0$ are not zero divisors
in $M_3(K[T])$ and $M_3(K[S])$. Using standard results (see
\cite{semalg}, Chapter 5), we get

\begin{corollary} \label{semipfactor3}
Semigroup algebras $K_0[M_0]$ and $K_0[M_1]$ are prime.
\end{corollary}

Due to the very special form of all reduced words in $C_3$ one can
also prove the following result.

\begin{theorem}\label{semiprimeC3}
     $K[C_3]$ is a semiprime algebra.
\end{theorem}

\begin{proof}
By $F$ we denote the free monoid generated by $a, b, c$. Write
$R=K[C_3]$ and suppose that  $x\in R$ is a nonzero element such
that $xRx=0$. Then $x$ can be uniquely written in the form
\begin{equation*}x= \sum\limits_{i\in I} \sigma_i u_i(cba)^{n_i}v_i+\sum\limits_{j\in J}\tau_j w_j(cab)^{l_j}z_j,\end{equation*}
for some disjoint index sets $I$, $J$, where $\sigma_i, \tau_j\neq
0$ are elements of the field $K$, $n_i, l_j\geqslant 0$,
$u_i\in\{1, a, ba\}$, $v_i\in\{1, c, cb\}$, $w_j\in\{1, b, ab\}$,
$z_j\in \{1, c, ca\}$. In $F$ consider the deg-lex order induced
by $a<b<c$. Let $u_0\alpha^mv_0$ ($\alpha \in\{cba, cab\}$) be the
leading term in the support of $x$. We may assume that its
coefficient is equal to $1$. Notice that $u_0$ and $v_0$ must be a
suffix, and a prefix respectively, of $\alpha$. Hence, there exist
words $p, q$ such that $pu_0=v_0q=\alpha$ holds in the free monoid
$F$. Then for all elements $w\neq u_0\alpha^mv_0$ in the support
of $x$ we have $\alpha^{m+2}=pu_0\alpha^mv_0q>pwq$ in $F$. If
$xRx=0$, then also $pxqRpxq=0$. In particular, $(pxq)^2=0$. On the
other hand, we know that in $K[C_3]$
\begin{equation*}
pxq=\alpha^{m+2}+\sum\limits_{y_i<\alpha^{m+2}}\rho_iy_i,
\end{equation*}
where $i\in (I\sqcup J)\setminus\{0\}$, $\rho_i\in K$, $y_i$ is
the reduced form of the word $pu_i(cba)^{n_i}v_iq$ if $i\in I$,
and if $i\in J$, then  $y_i$ is the reduced form of
$pw_i(cab)^{l_i}z_iq$. In particular, for every
$y_i\neq\alpha^{k+2}$ we have $y_i<\alpha^{m+2}$.

Since $\alpha^{m+2}\alpha^{m+2}$ has reduced form
$\alpha^{2(m+2)}$, for every $l,n\in I\sqcup J$ such that $(y_l,
y_n)\neq (\alpha^{m+2}, \alpha^{m+2})$ the reduced form $y_{ln}$
of the word $y_ly_n$ satisfies $y_{ln}<\alpha^{2(m+2)}$. In
particular, the leading term of  $(pxq)^2$ is equal to
$\alpha^{2(m+2)}$ and it is nonzero. This contradiction shows that
$K[C_3]$ is semiprime.
\end{proof}

\begin{remark} \label{semiprimeM0}
{\rm The argument used in the above proof also shows that for
every $n\geqslant 3$ the algebras of matrix type $K_{0}[M_{0}]$
and $K_{0}[M_{n-2}]$ defined for $K[C_n]$ are semiprime (actually,
prime). Indeed, $A_{0},B_{0}$ consist of all suffixes and
prefixes, respectively, of the word $x_{n}x_{n-1}\cdots x_{1}$.
While $A_{n-2}, B_{n-2}$ consist of all suffixes and prefixes,
respectively, of $x_{n}x_{1}\cdots x_{n-1}$, so the argument used
in the proof can be applied.}
\end{remark}

\subsection{Cycle of length $4$} For simplicity, we write $x_1=a$,
$x_2=b$, $x_3=c$, $x_4=d$. Recall that  $C_4$ has the following
presentation
\begin{gather*}
C_4=\langle a, b, c, d: a^2=a, b^2=b, c^2=c, d^2=d,ab=aba=bab,
bc=bcb=cbc,\\ cd=cdc=dcd, da=dad=ada, ac=ca, bd=db \rangle.
\end{gather*}
\noindent The form of the sets $A_0, B_0, A_1,B_1,A_2, B_2$
follows directly from  Theorem~\ref{tw1}.
\begin{lemma}\label{n=4}
    If an element of $C_4$ has a factor of the form  $w_0=dcba$,
    $w_1=dacb$ or $w_2=dabc$ then it is of the form  $\alpha_iw_i^k\beta_i$,
    with $k\geqslant1$, $\alpha_i\in A_i$, $\beta_i\in B_i$, where
    \begin{enumerate}
        \item $A_0=\{1, a, ba, cba\}$, $B_0=\{1, d, dc, dcb\}$;
        \item $A_1=\{1, b, cb, acb, ab, bacb\}$, $B_1=\{1, d, da, dac, dc, dacd\}$;
        \item $A_2=\{1, c, bc, abc\}$, $B_2=\{1, d, da, dab\}$.
    \end{enumerate}
\end{lemma}

 From Corollary~\ref{struktura} we know that $C_4$ has a chain of ideals
\begin{equation*}
\emptyset= I_2\s I_1\s I_0\s I_{-1},
\end{equation*}
 with semigroups of matrix type $M_0$, $M_1$ and $M_2$, coming from
 Corollary~\ref{wn3}, such that
\begin{enumerate}
    \item $M_2=(C_4dabcC_4)^0 \subseteq I_1/I_2$ and the set $(I_1/I_2)\setminus M_2$ is finite,
    \item $M_1=\{\alpha(dacb)^k\beta: \alpha\in A_1, \beta\in B_1, k\geqslant 1\}\cup\{\theta\}\subseteq I_0/I_1$
    and the set $(I_0/I_1)\setminus M_1$ is finite,
    \item $M_0=\{\alpha(dcba)^k\beta: \alpha\in A_0, \beta\in B_0, k\geqslant
    1\}\cup\{\theta\}=I_{-1}/I_0$,
    \item $C_4\setminus I_{-1}$ is finite, where $I_{-1}=I_0\cup
    C_4dcbaC_4$.
\end{enumerate}
We present these structures of matrix type below.
 A simple verification is left to the reader.
\begin{lemma}\label{lem19}
     Let $S_2$ be the cyclic semigroup generated by $s=dabc$. Then the ideal generated by  $s$ in $C_4$,
     with a zero adjoined, that is $M_2=(C_4dacbC_4)^{0}$, is a semigroup of matrix type
     $\mathcal{M}^{0}(S_2, A_2, B_2; P_2)$, where $A_2=\{1, c, bc, abc\}$, $B_2=\{1, d, da, dab\}$,
     with sandwich matrix (with coefficients in $S_2^1$)
    \begin{equation*}
    P_2= \begin{pmatrix}
    1 & 1 & 1 & 1\\
    1 & 1 & 1 & s\\
    1 & 1 & s & s\\
    1 & s & s & s
    \end{pmatrix}.
    \end{equation*}
\end{lemma}

\begin{lemma}\label{lem20}
     Let $S_1$ be the cyclic semigroup generated by  $s=dacb$. Then $M_1$
     is a semigroup of matrix type  $\mathcal{M}^{0}(S_1, A_1, B_1; P_1)$ where
     $A_1=\{1, b, cb, acb, ab, bacb\}$, $B_1=\{1, d, dc, dac, da, dacd\}$,
     with sandwich matrix (with coefficients in  $S_1^1\cup\{\theta\}$)
    \begin{equation*}
    P_1=
    \begin{pmatrix}
    1 & 1 & \theta & \theta & 1 & \theta\\
    1 & 1 & \theta & s & \theta & s\\
    \theta & \theta & \theta & s & s & s\\
    \theta & s & s & s & s & \theta \\
    1 & \theta & s & s & \theta & \theta\\
    \theta & s & s & \theta & \theta & s^2
    \end{pmatrix}.
    \end{equation*}
\end{lemma}

\begin{lemma}\label{lem21}
    Let $S$ be the cyclic semigroup generated by  $s=dcba$.
    Then $M_0$ is a semigroup of matrix type  $\mathcal{M}^{0}(S_0, A_0, B_0; P_0)$,
    where $s=dcba$, $A_0=\{1, a, ba, cba\}$, $B_0=\{1, d, dc, dcb\}$, with sandwich matrix
    \begin{equation*}
    P_0=\begin{pmatrix}
    1 & 1 & \theta & \theta\\
    1 & \theta & \theta & s\\
    \theta & \theta & s & s \\
    \theta & s & s & \theta
    \end{pmatrix}.
    \end{equation*}
\end{lemma}

We get the following consequence.
\begin{corollary}
Algebras $K_0[M_2]$, $K_0[ M_1]$ and $K_{0}[M_{0}]$ are algebras
of matrix type. Namely, $K_0[M_2]=\mathcal{M}(K[S_2], A_2, B_2;
P_2)$, $K_0[M_1]=\mathcal{M}(K[S_1], A_1, B_1; P_1)$ and
$K_0[M_0]=\mathcal{M}(K[S_0], A_0, B_0; P_0)$, where $S_2,
S_1,S_0$ are the cyclic semigroups generated by $s_2=dabc$,
$s_1=dacb$, and by  $s_0=dcba$, respectively.
\end{corollary}

A direct computation shows that $ \det P_2 = -(s_2-1)^3\neq 0.$
Similarly, one can see that $\det P_1=-s_1^3(s_1+1)^3\neq0$ and
$\det P_0=-s_0^2(s_0-1)\neq 0$, so that the matrices $P_i$ are not
zero divisors in the corresponding matrix rings $M_{n_i}(K[S_i])$,
for $i\in\{0, 1, 2\}$. Therefore, by \cite{semalg}, Chapter 5, we
get

\begin{corollary} \label{semipfactor4} Semigroup algebras $K_0[M_2]$, $K_0[M_1]$ and $K_0[M_0]$
are prime.
\end{corollary}

The assertion of Remark~\ref{semiprimeM0},
Corollary~\ref{semipfactor3} and Corollary~\ref{semipfactor4} will
be extended in Theorem~\ref{semiprime} to all algebras
$K_{0}[M_t]$ coming from monoids $C_n$, $n\geqslant 3$. This will
be crucial for our main results in
Section~\ref{Noetherian-chapter}. However, the proof is much more
complicated since the determinants of the corresponding matrices
cannot be easily computed.

\section{Noetherian Hecke--Kiselman algebras}\label{Noetherian-chapter}

In this section we characterize Noetherian Hecke--Kiselman
algebras $K[\HK_{\Theta}]$ of arbitrary oriented graphs $\Theta$.
The main difficulty is in proving that all algebras $K[C_n]$ are
Noetherian. We start, however, with describing some simple
obstacles to the Noetherian property. For $w,w'\in F$ we write $w
\sim w' $ if $w,w'$ represent the same element of $\HK_{\Theta}$.

\begin{lemma}  \label{acc}
\label{nonnoether} Let $\Theta$ be the graph obtained by adjoining
the arrow $y \rightarrow x_1$ to the cyclic graph $C_n $: $x_1
\rightarrow x_2 \rightarrow \cdots \rightarrow x_n \rightarrow
x_1$.  Then the monoid $\HK_{\Theta}$ does not satisfy the
ascending chain condition on left ideals, and it does not satisfy
the ascending chain condition on right ideals.
\end{lemma}

\begin{proof} Write $w_k =
(x_nx_{n-1} \cdots x_1)^k y$, for $k=1,2, \ldots $. It is clear
that $w_k$ cannot be rewritten in the monoid $\HK_{\Theta}$ except
for applying relations of the form $x_{i}^{2}=x_i, y^2=y$.
Therefore $w_k\notin \bigcup_{i=1}^{k-1}w_i\HK_{\Theta}$ for
$k\geqslant 2$. Hence, $\HK_{\Theta}$ does not satisfy acc on
right ideals.

Let $\phi : \langle x_1,\ldots ,x_n,y\rangle \longrightarrow
\langle x_1,\ldots ,x_n\rangle $ be the homomorphism such that
$\phi (w)$ is obtained from $w$ by erasing all occurrences of $y$.
Consider the following subsets of the free monoid $F=\langle
x_1,\ldots ,x_n,y \rangle$: $Z_{k} = \{ \phi^{-1}
((x_{1}^{i_{1}}x_{n}^{i_{n}} \cdots x_{2}^{i_{2}})^{k}) \mid
i_{j}\geqslant 1 \mbox{ for every } j\}$, for $k=1,2,\ldots $, and
\begin{eqnarray*} R_{k} =\{ wyvz \mid
 w\in \langle x_1,x_2 \rangle ,
   v\in \langle x_{2},y\rangle \langle x_{1}, y\rangle ,
    z\in Z_k, |wv|_{2} \geqslant 1\}.
\end{eqnarray*} We claim that $R_k$ is closed under relation
$\sim$. It is easy to see that $R_{k}$ is closed under $x\sim x^2$
and under $xz\sim zx$ for generators $ x,z$ not connected in the
graph $\Theta$ (the only such factors of a word $u\in R_{k}$ can
be of the form $yx_j, x_jy$, where $j\geqslant 2$). Moreover, $u$
does not have factors of the form $ x_jx_i$ with $i=3,\ldots ,n$
and $j\neq i+1$ (modulo $n$). So we do not have to consider
relations $x_ix_{i+1}x_i \sim x_{i+1}x_{i}x_{i+1} \sim x_ix_{i+1}$
for $i=2,\ldots, n$. It is also easy to see that every relation
$yx_{1}y\sim yx_1$, $x_{1}yx_{1}\sim yx_1$ and $x_{1}yx_{1}\sim
yx_1y$ leaves $R_{k}$ invariant. Finally, every relation $x_2
x_1x_2\sim x_1 x_2$, $x_1 x_2x_1\sim x_1x_2 $ and $x_1 x_2x_1\sim
x_2x_1x_2 $ leaves $R_{k}$ invariant. This proves the claim.

Define $v_{k}= x_{1}x_{2}y(x_1x_n\cdots x_2)^k$, for $k=1,2,\ldots
$. Notice that $v_{k}\in R_k$ but $v_{k}\notin FR_{i}$ for $i<k$.
It follows that  $v_k \notin \bigcup_{i=1}^{k-1}\HK_{\Theta}
v_{i}$, for every $k\geqslant 2$. Therefore $\HK_{\Theta}$ does
not satisfy acc on left ideals.
\end{proof}

We continue with a useful observation of independent interest. In
order to avoid confusion, we denote by $[w]$ the equivalence class
of a word $w$ in $F$, with respect to the equivalence relation
$\sim$.

\begin{lemma} \label{induction}
Assume that $y_1, y_2, \ldots, y_{n-1}$ are the consecutive
vertices of a  cyclic graph $C_{n-1}$. Consider an epimorphism
$\phi$ from the free monoid $Y=\langle y_1, \ldots, y_{n-1}
\rangle$ to the submonoid $\langle  x_2, \ldots, x_{n-1},
x_{n}x_{1} \rangle$ of $F$ defined by
$$\phi(y_i) = \begin{cases} x_{i+1},& \text{ for } 1 \leqslant i \leqslant n-2,\\
                         x_nx_1 ,& \text{ for } i = n-1.\end{cases}$$
Then $\phi$ induces a homomorphism $\overline{\phi}: C_{n-1}
\longrightarrow C_n$ given by the formula $\overline{\phi}([w]) =
[\phi(w)]$, for every $w \in \langle y_1, \ldots, y_{n-1}\rangle$.
Moreover, $\overline{\phi}$ determines an isomorphism
$C_{n-1}\cong \langle  x_2,\ldots, x_{n-1}, x_nx_1\rangle
\subseteq C_n$.
\end{lemma}
\begin{proof} It is verified in \cite{mecel_okninski1}, Lemma~4,
that $\overline{\phi}$ is a homomorphism. We claim that if a word
$w=w(y_1,\ldots, y_{n-1})$  is reduced in the sense of the
reduction system $S'$ as in Theorem~\ref{basisCn}, defined with
respect to the deg-lex order extending $y_1<\cdots < y_{n-1}$ in
the free monoid $Y=\langle y_1,\ldots, y_{n-1}\rangle $, then the
word $w(  x_2,\ldots , x_{n-1}, x_nx_1)$ is reduced with respect
to the system $S$ in the free monoid $F=\langle x_1,\ldots,
x_{n}\rangle $.

If $w\in Y$ then it is clear that if $\phi (w)$  contains a factor
(that is the leading term of a reduction) of type (1) in
Theorem~\ref{basisCn}, then also $w$ contains such a factor.
Assume that $\phi (w)$ contains a factor $x_jx_i$ of type (2).
Then $w$ contains a factor $y_{j-1}y_{i-1}$. Assume that $\phi
(w)$ has a factor  $x_i ux_i$ that is of type (4) or (5). If $i=1$
or $i=n$ then $\phi (w)$ has a factor $x_nx_1 vx_nx_1$. If $v$
does not contain $x_2$ ($x_{n-1}$, respectively) then
$\phi^{-1}(v)$ does not contain $y_1$ ($y_{n-2}$, respectively),
and we are done. If $i\neq 1,n$, and $u$ does not contain
$x_{i+1}$ ($x_{i-1}$, respectively) then $\phi^{-1}(u)$ does not
contain $y_{i}$ (respectively, $y_{i-2}$ if $i>2$; and if $i=2$
then $\phi^{-1}(u)$ does not contain $y_{n-1}$), as desired.
Assume that $\phi (w)$ contains a factor of the form
$x_n(x_1\dotsm x_i)x_j$ for $i,j\in\{1,\dotsc,n\}$ such that
$i+1<j<n-1$. Then $w$ contains a factor $y_{n-1}(y_1 \cdots
y_{i-1})y_{j-1}$ or $y_{n-1}y_{j-1}$, and the assertion follows as
well. This proves the claim.

Therefore $\overline{\phi}$ is injective. The result follows.
\end{proof}

 The crucial step in the proofs of the main results of this
section is based on the following observation.

\begin{proposition}  \label{anihilator}
Assume that $t\in \{1,\ldots ,n-3\}$ and $\alpha \in K_{0}[M_t]$
is such that $\alpha x_i=0$ in $K_{0}[M_t]$ for every
$i\in\{1,\ldots, n\}$. Assume also that every $w\in \supp(\alpha)$
is of the form $(x_nq_t)^kb$, where $k\geqslant 1$ and $b\in B_t$.
Then $\alpha =0$.
\end{proposition}

In order to prove it, we need some preparatory technical lemmas.
We assume that $t\in\{1,\ldots, n-3\}$. Moreover, we will suppose
that a nonzero $\alpha \in K_{0}[M_t]$ is given that satisfies the
hypotheses of the proposition. The aim is to come to a
contradiction.

Roughly speaking, the first lemma describes the reduced form of
any word of type $wx_r$ for $w$ in block form (see the convention
introduced after Theorem~\ref{tw1}), namely
\begin{equation*}
(x_nq_t)x_n(x_{1}\cdots x_{i_1}x_{n-1}\cdots x_{j_1})\cdots
x_n(x_{1}\cdots x_{i_k}x_{n-1}\cdots x_{j_k}),
\end{equation*}
with certain conditions on indices, and $x_r$ such that
$n-1\geqslant r\geqslant j_k-1$ or $r\leqslant i_k+1$. This means
that $x_r$ cannot be pushed to the left by using only reductions
(2) or (3) in such a way that $wx_r=(x_nq_t)x_n(x_{1}\cdots
x_{i_1}x_{n-1}\cdots x_{j_1})\cdots x_n x_r(x_{1}\cdots
x_{i_k}x_{n-1}\cdots x_{j_k})$ in $C_n$.

\begin{lemma}\label{multipl}
    Let $t\in\{1,\ldots, n-3\}$. Consider the word $w$ of the form
    \begin{equation*}
    w=(x_nq_t)x_n(x_{1}\cdots x_{i_1}x_{n-1}\cdots x_{j_1})\cdots x_n(x_{1}\cdots x_{i_k}x_{n-1}\cdots x_{j_k}),
    \end{equation*}
    where $i_k<i_{k-1}<\cdots<i_1<t+1<j_1<\cdots<j_k$. The word $x_nq_t$ is also
    assumed to be of the above type for $k=0$. Moreover, in every $w$ we use the convention that $i_0=t, j_0=t+1$.
     Let  $r\geqslant t$ be such that $n-1\geqslant r\geqslant j_k-1$ or $r\leqslant i_k+1$
      (so in the latter case $t\leqslant r\leqslant i_k+1\leqslant t+1$).
     Then the following holds:
    \begin{enumerate}
        \item if $n-1\geqslant r> j_k$, then $wx_r\in I_t$;
        \item if $r=j_k$, then $wx_r=w$ in $C_n$;
        \item if $j_k=r+1$, then either $wx_r=w$ in $C_n$ or the word $wx_r$ is reduced;
        \item if $j_k>r+1$, $r=t$, $i_k=t-1$, then either (for $k=1$) the word $wx_r$ has the reduced form
        $(x_nq_t)x_nx_1\cdots x_tx_{n-1}\cdots x_{j_k}$ or (for $k\geqslant 2$) $wx_r\in I_t$;
        \item if $j_k>r+1$, $r=t$, $i_k=t$, then $wx_r=w$ in $C_n$;
        \item if $j_k>r+1$, $r=t+1$, $i_k=t$, then $wx_r\in I_t$.
    \end{enumerate}
\end{lemma}
\begin{proof}
    Parts 2. and 5. are clear.

To prove part 1., we proceed by induction on $k$ (the number of
blocks in the word $w$).  Let $n-1\geqslant r>t+1$. If $k=0$ then
$w=x_nq_t$ and
        \begin{align*}
        wx_r&\arr{(5r)}x_nx_1\cdots x_tx_{n-1}\cdots x_{r+1}x_{r-1}\cdots x_{t+1}x_r\arr{(2)}x_nx_1\cdots x_t(x_{r-1}\cdots x_{t+1})(x_{n-1}\cdots x_{r})\\
       & \arr{(3)}(x_{r-1}\cdots x_{t+2})(x_nx_1\cdots x_{t+1})(x_{n-1}\cdots
        x_{r}) .
        \end{align*}
        From Lemma \ref{lem13} we obtain $(x_{r-1}\cdots x_{t+2})(x_nx_1\cdots x_{t+1})(x_{n-1}\cdots x_{r})\in I_t$,
        as desired.\\
        So, assume that the assertion holds for every $m<k$, where $k\geqslant 1$. Consider
        \begin{equation*}
        wx_{r}=(x_nq_t)x_n(x_{1}\cdots x_{i_1}x_{n-1}\cdots x_{j_1})\cdots x_n(x_{1}\cdots x_{i_k}x_{n-1}\cdots x_{j_k})x_r,
        \end{equation*}
        for $r>j_k$. Then we have
        \begin{align*}
        x_n(x_{1}\cdots x_{i_k}x_{n-1}\cdots x_{j_k})x_{r}& \arr{(5r)}x_nx_1\cdots x_{i_k}x_{n-1}\cdots
        x_{r+1}x_{r-1}\cdots x_{j_k}x_r \\ &\arr{(2)}x_nx_1\cdots x_{i_k}x_{r-1}\cdots x_{j_k}x_{n-1}\cdots x_{r}
        \end{align*}
        From the assumptions we know that $j_{k}>i_{k} +1$  and $r-1<n-1$, so the following holds:
        \begin{equation*}
        x_nx_1\cdots x_{i_k}x_{r-1}\cdots x_{j_k}x_{n-1}\cdots x_{r}\arr{(3)}x_{r-1}\cdots x_{j_k}x_nx_1\cdots x_{i_k}x_{n-1}\cdots x_r.
        \end{equation*}
        By the assumptions $j_k<r\leqslant n-1$ and $j_{k-1}<j_k$ so $j_{k-1}<r-1\leqslant n-1$. $I_t$ is an ideal in $C_n$,
        so from the above calculation and the induction hypothesis
        for the element
        \begin{equation*}
        v=(x_nq_t)x_n(x_{1}\cdots x_{i_1}x_{n-1}\cdots x_{j_1})\cdots x_n(x_{1}\cdots x_{i_{k-1}}x_{n-1}\cdots x_{j_{k-1}})
        \end{equation*} (a word with $k-1$ blocks) the following holds in $C_n$
        \begin{equation*}
        wx_r=vx_{r-1}\cdots x_{j_k}x_nx_1\cdots x_{i_k}x_{n-1}\cdots x_r\in
        I_t.
        \end{equation*}
        Hence part 1. follows.

         To prove part 3., assume that $j_k=r+1$. Recall that $i_0=t$ and $j_0=t+1$.
         It follows that for $k=0$ we have $r=t$. In this case
         $(x_nq_t)x_r\arr{(4t)}w$.
         Hence, we can assume that $k\geqslant 1$. Then \begin{equation*}
        wx_r=(x_nq_t)x_n(x_{1}\cdots x_{i_1}x_{n-1}\cdots x_{j_1})\cdots x_n(x_{1}\cdots x_{i_k}x_{n-1}\cdots
        x_{j_k})x_{j_k-1} .
        \end{equation*}
        If $j_{k-1}<j_k-1$, then we see that the above word is reduced.\\
        Hence, assume now that $j_{k-1}=j_k-1$.
        Then the word $wx_r$ has a factor $x_{j_{k}-1}x_nx_1\cdots x_{i_k}x_{n-1}\cdots x_{j_{k}}x_{j_{k}-1}$. If $i_k+1<j_k-1$,
        then
        \begin{equation*}x_{j_{k}-1}x_nx_1\cdots x_{i_k}x_{n-1}\cdots
        x_{j_{k}}x_{j_{k}-1}\arr{(4(j_k-1))}x_{j_{k-1}}x_nx_1\cdots x_{i_k}x_{n-1}\cdots x_{j_{k}}.\end{equation*}
        It follows that $wx_r=w$.\\
        Finally, if $i_k+1\geqslant j_k-1$, then $i_k<j_{k-1}=j_k-1\leqslant i_k+1$, so that $j_{k-1}= i_{k}+1$.
        Hence $i_{k} \leqslant i_{k-1} < t+1 \leqslant j_{k-1}$ implies that
        $i_{k-1}=i_k=t$,
        $j_{k}=t+2$.
        It follows that $wx_r$ is reduced. This proves part 3.

        In the proof of the remaining assertions (parts 4. and 6.) we can assume that $k\geqslant 1$,
        because for $k=0$ it is impossible to have $t+1=j_k>r+1$ and $r\in \{
        t,t+1\}$.

         To prove part 4., assume that $j_k>r+1$, $r=t$, $i_k=t-1$. Then from the definition of $w$
        we obtain that $k\in\{1, 2\}$ and either
        $w=(x_nq_t)x_nx_1\cdots x_{t}x_{n-1}\cdots x_{j_1}x_nx_1\cdots x_{t-1}x_{n-1}\cdots x_{j_2}$,
        where $j_2>j_1>t+1$ or $w=(x_nq_t)x_nx_1\cdots x_{t-1}x_{n-1}\cdots x_{j_1}$.
        In the first case
        \begin{align*}wx_t&\arr{(2)}(x_nq_t)x_nx_1\cdots x_{t}x_{n-1}\cdots x_{j_1}x_nx_1\cdots x_{t}x_{n-1}\cdots x_{j_2}\\
       & \arr{(5t)}\cdots \arr{(51)}\arr{(5n)}(x_nq_t)x_{n-1}\cdots x_{j_1}x_nx_1\cdots x_{t}x_{n-1}\cdots x_{j_2}.\end{align*}
        From part 1. applied to $x_nq_t$ and $r=n-1$ we get $wx_r\in I_t$.\\
        In the second case $wx_t\arr{(2)}(x_nq_t)x_nx_1\cdots x_tx_{n-1}\cdots x_{j_1}$ and the last word
        is reduced.

        To prove part 6., assume that $j_k>r+1$, $r=t+1$, $i_k=t$. Then from the definition of $w$ it follows that $k=1$ and
        \begin{align*}
        wx_{t+1}&\arr{(2)}(x_nq_t)x_nx_1\cdots x_{t+1}x_{n-1}\cdots x_{j_1}\arr{(5(t+1))}x_nx_1\cdots x_tx_{n-1}\cdots
        x_{t+2}x_nx_1\cdots x_{t+1}x_{n-1}\cdots x_{j_1}\\
       & \arr{(5t)}\cdots\arr{(51)}\arr{(5n)}x_{n-1}\cdots x_{t+2}x_nx_1\cdots x_{t+1}x_{n-1}\cdots x_{j_1}\in I_t
        \end{align*}
        by Lemma~\ref{lem13}.
\end{proof}
    We continue under the assumptions of Proposition~\ref{anihilator}.
    By Theorem~\ref{tw1}, every $w\in \supp (\alpha)$ must satisfy one of the following conditions:
    \begin{enumerate}
        \item[(i)] $x_nx_1\cdots x_{i_{s-1}}x_{n-1}\cdots x_{j_{s-1}}x_nx_1\cdots x_{i_s}x_{n-1}\cdots x_{j_s}\in \suff(w)$,
        where $i_{s}<i_{s-1}<t+1<j_{s-1}<j_{s}\leqslant n-1$, or $i_s=i_{s-1}=t$
        and $j_{s-1}=t+1<j_{s}$,
        \item[(ii)] $x_nx_1\cdots x_{i_{s-1}}x_{n-1}\cdots x_{j_{s-1}}x_nx_1\cdots x_{i_s}x_{n-1}
        \cdots x_{j_s}x_nx_1\cdots x_{i_{s+1}}\in \suff(w)$,
        where \\
        $i_{s+1}<i_{s}<i_{s-1}<t+1<j_{s-1}<j_{s}\leqslant n-1$, \\
        or $i_s=i_{s-1}=t>i_{s+1}$ and $j_{s-1}=t+1<j_{s}$;\\
        or $(x_nq_t)x_nx_1\cdots x_{i_{s+1}}\in \suff(w)$ with $i_{s+1}\leqslant
        t$,
        \item[(iii)] $x_nx_1\cdots x_{i_{s-1}}x_{n-1}\cdots x_{j_{s-1}}x_nx_{n-1}\cdots x_{j_s}\in \suff (w)$,
        where $2\leqslant i_{s-1}<t+1<j_{s-1}<j_{s}\leqslant n$
        \item[(iv)] $x_nx_1x_{n-1}\cdots x_{j_{s-1}}x_nx_{n-1}\cdots x_{j_{s}}\in \suff(w)$,
        where $2\leqslant t+1<j_{s-1}<j_s\leqslant n$,
        \item[(v)] $b=x_nx_{n-1}\cdots x_{j_{s}}$, where $t+1<j_s\leqslant n$,
        \item[(vi)]$b=1$, i.e. $w=(x_nq_t)^k$.
    \end{enumerate}

Hence, we can write
$\alpha=\alpha_{(i)}+\alpha_{(ii)}+\alpha_{(iii)}+\alpha_{(iv)}+\alpha_{(v)}+\alpha_{(vi)}$,
where $\supp(\alpha_k)$ consists of all words of the form $(k)$
listed above, which are in the support of the element $\alpha$. We
will prove that for every $k\in\{(i),\ldots, (vi)\}$ the element
$\alpha_{(k)}$ is zero,
which will contradict the supposition that $\alpha\neq 0$. \\
First, we prove the following result concerning $\alpha x_1$.

\begin{lemma} \label{6types}
    Let $\alpha$ be as described above. Then
    \begin{enumerate}
        \item $\alpha_{(vi)}=0$;
        \item $\alpha_{(i)}=\alpha_{(i), i_s=1}$;
        \item $\alpha_{(ii)}=\alpha_{(ii), i_{s+1}=1}$,
    \end{enumerate}
where $\alpha_{(i)} = \alpha_{(i),i_s=1}+ \alpha_{(i), i_s>1}$ and
$\supp(\alpha_{(i), i_s=1})$ consists of all words from the
support of $\alpha_{(i)}$ with $i_s=1$, while $\supp (\alpha_{(i),
i_s>1})$ does not contain such words; similarly $\alpha_{(ii),
i_{s+1}=1}$ involves all words from the support of $\alpha_{(ii)}$
with $i_{s+1}=1$ (see the description of $\alpha_{(i)}$,
$\alpha_{(ii)}$).
\end{lemma}
\begin{proof} We know that $\alpha x_1=0$ in $K_0[M_t]$. We calculate the reduced forms of $wx_1$
for all $w\in \supp (\alpha_{k})$, for $k\in\{(i),\ldots, (vi)\}$.
It will be more convenient to consider certain suffixes
of the given word $w$.
\begin{itemize}
    \item $x_nx_1\cdots x_{i_s}x_{n-1}\cdots x_{j_s}x_1
    \arr{(41)}x_nx_1\cdots x_{i_s}x_{n-1}\cdots x_{j_s}$,
     so $\alpha_{(i)}x_1=\alpha_{(i)}$;
    \item $x_nx_1\cdots x_{i_{s+1}}x_1\arr{(41)}x_nx_1\cdots x_{i_{s+1}}$, whence $\alpha_{(ii)}x_1=\alpha_{(ii)}$;
    \item $x_nx_1\cdots x_{i_{s-1}}x_{n-1}\cdots x_{j_{s-1}}x_nx_{n-1}\cdots x_{j_s}x_1\arr{(*)}
    \begin{cases}x_nx_1\cdots x_{i_{s-1}}x_{n-1}\cdots x_{j_{s-1}}x_nx_1
   \textrm{ for }j_{s}=n\\x_nx_1\cdots x_{i_{s-1}}x_{n-1}\cdots x_{j_{s-1}}x_nx_1x_{n-1}\cdots x_{j_s}\textrm{ for
   }j_s<n ,
    \end{cases}$
     where $(*)$ denotes equality in the first case and reduction $(2)$ in the second case.
    We see that in the first case $(j_s=n)$ the obtained word is reduced of type $(ii)$ with $i_{s+1}=1$.
    In the second case $(j_s<n)$ the word is reduced of type $(i)$ with $i_s=1$.
    \item $\displaystyle\begin{aligned}[t] x_nx_1x_{n-1}\cdots
    x_{j_{s-1}}x_nx_{n-1}\cdots x_{j_{s}}x_1 & \arr{(2)}x_nx_1x_{n-1}
    \cdots x_{j_{s-1}}x_nx_1x_{n-1}\cdots x_{j_{s}}\\
    & \arr{(51), (5n)}x_{n-1}\cdots x_{j_{s-1}}x_nx_1x_{n-1}\cdots
    x_{j_{s}}.\end{aligned}$

In this case the obtained form of the element $wx_1$ has a factor
of the form
    \begin{equation*}
     (x_nq_t)x_n(x_{1}\cdots x_{i_1}x_{n-1}\cdots x_{j_1})\cdots x_n(x_{1}\cdots x_{i_{k}}x_{n-1}\cdots x_{j_{k}})x_{n-1},
    \end{equation*} where $k=s-2$, $j_k<n-1$ (notice that $s\geqslant 2$). Assertion 1. of
    Lemma \ref{multipl} implies that $wx_1=0$ in $K_{0}[M_t]$ for every $w\in\supp(\alpha_{(iv)})$.
    \item $x_nx_{n-1}\cdots x_{j_{s}}x_1\arr{(*)}\begin{cases}
    x_nx_{n-1}\cdots x_{j_{s}}x_1\textrm{ for }j_s=n\\x_nx_1x_{n-1}\cdots x_{j_{s}} \textrm{ for
    }j_s<n ,
    \end{cases}$\\
    where $(*)$ denotes equality in the first case and reduction $(2)$ in the second case.
    We see that in the first case the  word $wx_1$ is of the reduced form $(ii)$ with $i_{s+1}=1$,
    whereas in the second case we obtain a reduced word of type $(i)$ with $i_s=1$.
    \item $(x_nq_t)x_1\arr{(41)}x_nq_t$, so $\alpha_{(vi)}x_1=\alpha_{(vi)}$.
\end{itemize}
From the above calculations we see that in $K_{0}[M_t]$
\begin{equation*}0=(\alpha_{(ii)}+\alpha_{(iii),
j_s=n}x_1+\alpha_{(v),j_s=n}x_1)+(\alpha_{(i)}+\alpha_{(iii),
j_s<n}x_1+\alpha_{(v), j_s<n}x_1)+\alpha_{(vi)}.
\end{equation*}
It is clear that the terms from $\alpha_{(vi)}$ are the only terms
of type $(vi)$ in the above sum, so $\alpha_{(vi)}=0$. Moreover,
reduced forms of elements from $\alpha_{(i)}+\alpha_{(iii),
j_s=n}x_1+\alpha_{(v),j_s=n}x_1$ are of type $(i)$, whereas
reduced forms of words in the sum $\alpha_{(ii)}+\alpha_{(iii),
j_s<n}x_1+\alpha_{(v), j_s<n}x_1$ are of type $(ii)$. It follows
that these sums are $0$ in $K_{0}[M_t]$. It is not difficult to
see that every word from $\supp(\alpha_{(iii), j_s=n}x_1)$ and
$\supp(\alpha_{(v),j_s=n}x_1)$ has reduced form ending with
$x_nx_1$, so $\alpha_{(ii)}=\alpha_{(ii),i_{s+1}=1}$. Similarly,
every (reduced) word from $\supp(\alpha_{(iii), j_s<n}x_1)$ and
$\supp(\alpha_{(v), j_s<n}x_1)$ has a suffix of the form
$x_nx_1x_{n-1}\cdots x_j$ for some $j$, so
$\alpha_{(i)}=\alpha_{(i), i_s=1}$.\end{proof}

It follows that $\supp (\alpha) =
\supp(\alpha_{(i),i_s=1})\cup\supp(\alpha_{(ii),
i_{s+1}=1})\cup\supp(\alpha_{(iii)})\cup\supp(\alpha_{(iv)})\cup\supp(\alpha_{(v)})$.
Let $m=\textrm{min}\{j_s :w\in\supp(\alpha)\}$, with $j_s$ defined
for every word $w$ as in cases (i)-(vi) listed before
Lemma~\ref{6types}. Then $n\geqslant m\geqslant t+1\geqslant2$. By
our assumption, also $\alpha x_{m-1}=0$. We calculate the reduced
form of words $wx_{m-1}$, where $w\in \supp(\alpha)$. By $s_k$ we
mean an appropriately chosen suffix of the word from the support
of $\alpha_{k}$.
We consider the following two cases.\\
{\bf Case I.} Assume that $m=j_{s}$.
\begin{enumerate}
    \item First, suppose that $j_{s-1}=j_s-1$. Then
    \begin{enumerate}
        \item   $s_{(i)}x_{m-1}=x_nx_1\cdots x_{i_{s-1}}x_{n-1}\cdots x_{j_{s}-1}x_nx_1\cdots x_{i_s}x_{n-1}\cdots x_{j_s}x_{j_{s}-1}$\\
        If $i_s+1<j_{s}-1$, then $s_{(i)}x_{m-1}\arr{(4 j_s)}s_{(i)}$.\\
        Otherwise we have $i_s+1=j_{s}-1$, which implies that $j_s-2\leqslant i_{s-1}<j_s-1$, so it follows easily that $i_{s-1}=t, j_{s-1}=t+1$.
        In this case $s_{(i)}x_{m-1}=x_nq_tx_nx_1\cdots x_tx_{n-1}\cdots
        x_{t+1}=(x_nq_t)^2$.
        \item $s_{(ii)}x_{m-1}=x_nx_1\cdots x_{i_{s-1}}x_{n-1}\cdots x_{j_{s-1}}x_nx_1\cdots x_{i_s}x_{n-1}\cdots x_{j_s}x_nx_1x_{j_s-1}$.\\
         Suppose that $j_{s-1}+1=j_s\leqslant 3$. Since $j_{s-1}\geqslant t+1$, it follows that $t=1$ and $s_{(ii)}$
         must be the word
         $x_nx_1x_{n-1}\cdots x_2x_nx_1x_{n-1}\cdots x_{3}x_nx_1$, which is not reduced.
         Therefore we can assume that $n-1\geqslant j_{s}>3$ and
        \begin{equation*}s_{(ii)}x_{m-1}\arr{(3)}x_nx_1\cdots x_{i_{s-1}}x_{n-1}\cdots x_{j_{s-1}}x_nx_1
        \cdots x_{i_s}x_{n-1}\cdots x_{j_s}x_{j_s-1}x_nx_1.\end{equation*}
        It is clear that the reduced word $x_nx_1\cdots
        x_{i_{s-1}}x_{n-1}\cdots x_{j_{s-1}}x_nx_1\cdots
        x_{i_s}x_{n-1}\cdots x_{j_s}$ is of the form $(ii)$. From the
        previous case we obtain
        \begin{equation*}s_{(ii)}x_{m-1}=\begin{cases}s_{(ii)}\textrm{ if }
        i_s+1<j_{s}-1\\(x_nq_t)^{2}x_nx_1\textrm{ otherwise }.
        \end{cases}\end{equation*}
        \item $s_{(iii),(iv),(v)}x_{m-1}=x_{j_{s}-1}x_nx_{n-1}\cdots x_{j_s}x_{j_{s}-1}\arr{(4 (j_{s}-1))}x_{j_{s}-1}x_nx_{n-1}\cdots x_{j_s}$.
        It follows that for every $w$ of the form $(iii)$, $(iv)$ or $(v)$ we have $wx_{m-1}=w$.
    \end{enumerate}
    \item Secondly, assume that $j_{s-1}<j_s-1$. Then
     \begin{enumerate}
        \item for every $w$ of the form $(i)$ $wx_{m-1}$ is reduced;
        \item since $j_s-1>j_{s-1}\geqslant2$, then
        \begin{align*}
        s_{(ii), i_{s+1}=1}x_{m-1} &=x_nx_1\cdots x_{i_{s-1}}x_{n-1}\cdots x_{j_{s-1}}x_nx_1\cdots x_{i_s}x_{n-1}\cdots x_{j_s}x_nx_1x_{j_s-1}\\
        & \arr{(3)}x_nx_1\cdots x_{i_{s-1}}x_{n-1}\cdots x_{j_{s-1}}x_nx_1\cdots x_{i_s}x_{n-1}\cdots x_{j_s}x_{j_s-1}x_nx_1.
        \end{align*}
        It follows that the reduced form of $wx_{m-1}$, where $w\in \supp(\alpha_{(ii),
        i_{s+1}=1})$, has a suffix as above;
        \item similarly, it is clear that $w x_{m-1}$ is reduced for every $w\in \supp(\alpha_{k})$,
        where $k\in \{(iii), (iv), (v)\}$.
     \end{enumerate}
 \item  Assume that $s_{(ii)}x_{m-1}=(x_nq_t)x_nx_1x_{m-1}$. In this case $m=j_s=t+1$. If
 $t=1$, then $s_{(ii)}x_{t}=s_{(ii)}$ in $C_{n}$. Moreover, if $t\geqslant 3$ then $(x_nq_t)x_nx_1x_t\arr{(4t)}(x_nq_t)x_nx_1$,
 so also $s_{(ii)}x_t=s_{(ii)}$. Finally, if $t=2$, it is easy to see that $wx_t$ is in the reduced form.
\end{enumerate}
We summarize the foregoing observations as follows.
\begin{corollary}\label{m=j_s}
    Let $m=j_s$ be as described above. Consider an element $w$ from the
    support of $\alpha_{(i),i_s=1}$, $\alpha_{(iii)}$, $\alpha_{(iv)}$ or
    $\alpha_{(v)}$.
    \begin{enumerate}
    \item[1.] If $j_{s-1}=j_s-1$, then $wx_{m-1}=w$ in $C_n$ or $wx_{m-1}$ is of the form $(vi)$.
    \item[2.] If $j_{s-1}<j_s-1$, then $wx_{m-1}$ is reduced.
    \end{enumerate}
Assume now that $w$ is in the support of $\alpha_{(ii),
i_{s+1}=1}$.
\begin{enumerate}
    \item[3.] If $(x_nq_t)x_nx_1\in\suff(w)$, then either (for $t\neq 2$) $wx_{m-1}=w$ in $C_n$ or $wx_{m-1}$
    has the reduced form $(x_nq_2)^Mx_nx_1x_2$, for some $M\geqslant
    1$.
    \item[4.] If $j_{s-1}=j_s-1$, then either $wx_{m-1}=w$ in $C_n$ or $wx_{m-1}$ is of the form $(x_nq_t)^Mx_nx_1$,
    where $M\geqslant 1$.
    \item[5.] If $j_{s-1}<j_s-1$, then $w=vx_nx_1$ for some reduced word $v$ and $wx_{m-1}$ has the reduced form $vx_{m-1}x_nx_1$.
\end{enumerate}
In particular, words from the supports of $\alpha_{(i),i_s=1}$,
$\alpha_{(ii), i_{s+1}=1}$,  $\alpha_{(iii)}$, $\alpha_{(iv)}$ and
$\alpha_{(v)}$ multiplied by $x_{m-1}$ have reduced forms ending
with $x_{n-1}\ldots x_{j}$, $x_{j}x_nx_1$  or $x_{j}x_nx_1x_2$,
where $j\leqslant m$.
\end{corollary}
\noindent {\bf Case II.} Now assume that $m<j_s$. In particular
$m-1<n-1$.

We claim that if $w$ is a word in the support of
$\alpha_{(i),i_s=1}$, $\alpha_{(iii)}$, $\alpha_{(iv)}$ or
$\alpha_{(v)}$, then $wx_{m-1}$ is $0$ in $K_{0}[M_t]$ or its
reduced form has a suffix of the form $x_{n-1}\cdots x_j$ for some
$j>m$. Moreover, if $w$ is in
$\supp(\alpha_{(ii), i_{s+1}=1})$, then $wx_{m-1}$ is $0$ or $\suff(wx_{m-1})=x_{n-1}\cdots x_jx_nx_1$ for some $j>m$.\\
The idea is to reduce words by pushing $x_{m-1}$ to the left and
then to use Lemma \ref{multipl}.  As before, by $w_k$ we denote a
suffix of a word of type $(k)$.
\begin{enumerate}
\item[(a)] $w_{(i)}x_{m-1}=(x_nq_t)x_n(x_{1}\cdots x_{i_1}x_{n-1}\cdots
x_{j_1})\cdots x_n(x_{1}\cdots x_{i_s}x_{n-1}\cdots
x_{j_s})x_{m-1}$. As long as $j_k-1>m-1>i_k+1$ ($k=1,\ldots, s+1$)
we use reductions $(2)$ and $(3)$ to push $x_{m-1}$ to the left.
After this procedure we obtain a word with a prefix $vx_{m-1}$,
where $v$ is exactly a word from Lemma \ref{multipl}, for some
$k_0$ and $r=m-1$. By the assumption $j_s\neq m$ (hence, it is
impossible that $k_0=s$ and $j_{k_0}=r+1$), so applying Lemma
\ref{multipl} we obtain that $w_{(i)}x_{m-1}$ is either in $I_t$
or its reduced form ends with $x_{n-1}\cdots x_{j_s}$, $j_s>m$.
\item [(b)] Since $t+1\leqslant m<j_s$, we must have $w\neq(x_nq_t)^Mx_nx_1$. Then \begin{align*}
w_{(ii)}x_{m-1}=x_nx_1\cdots x_{i_{s-1}}x_{n-1}\cdots
x_{j_{s-1}}x_nx_1\cdots x_{i_s}x_{n-1}\cdots x_{j_s}x_nx_1x_{m-1}.
\end{align*}
If $m-1=1$, then  $w_{(ii)}x_{m-1}=w_{(ii)}$ in $C_n$, and hence it has a suffix $x_{n-1}\cdots x_{j_s}x_nx_1$, $j_s>m$.\\
If $m-1=2$, then $j_s>3$ and $t\leqslant 2$. From the form of
$w_{(ii)}$ we see that in this case $i_s>1$ and of course
$i_s\leqslant 2$. It follows that $i_s=2$. Then
\begin{equation*}
w_{(ii)}x_{m-1}\arr{(42), (41), (4n)}x_nx_1\cdots
x_{i_{s-1}}x_{n-1}\cdots x_{j_{s-1}}x_{n-1}\cdots
x_{j_s}x_nx_1x_{2}.
\end{equation*}
Applying assertion 1. of Lemma \ref{multipl}, it follows that this word is in $I_t$.\\
If $m-1>2$ then $s_{(ii)}x_{m-1}\arr{(3)}x_nx_1\cdots
x_{i_{s-1}}x_{n-1}\cdots x_{j_{s-1}}x_nx_1\cdots
x_{i_s}x_{n-1}\cdots x_{j_s}x_{m-1}x_nx_1$. Using the observation
made in the previous case and Lemma \ref{multipl},  we get that
either $w_{(ii)}x_{m-1}\in I_t$ or its reduced form has a suffix
$x_{n-1}\cdots x_{j_s}x_nx_1$, for $j_s>m$.
\item[(c)] Every word $w\in\supp(\alpha_{(iii)})\cup\supp(\alpha_{(iv)})\cup\supp(\alpha_{(v)})$ can be written as
 $w=vx_nx_{n-1}\cdots x_{j_s}$, where $v$ has a block form
as in Lemma \ref{multipl}. Then
$wx_{m-1}\arr{(2)}vx_{m-1}x_nx_{n-1}\ldots x_{j_s}$. Pushing
$x_{m-1}$ to the left by using reductions (2) and (3) we can apply
Lemma \ref{multipl}. It follows that either $wx_{m-1}\in I_t$ or
its reduced form has a suffix $x_nx_{n-1}\cdots x_{j_s}$, for
$j_s>m$.
\end{enumerate}
This completes the proof of our claim in Case II.\\

By our assumptions (of Proposition~\ref{anihilator}), we know that
$\alpha x_{m-1}=0$ in $K_{0}[M_t]$. From the above discussion it
follows that for every  $w\in \supp (\alpha)$ either $wx_{m-1}$ is
$0$ (and it is possible only if $m<j_s$) or a suffix of the
reduced form of $wx_{m-1}$ is equal to $x_jx_nx_1$, $x_jx_nx_1x_2$
(only if $w\in \supp (\alpha_{(ii)}))$, or to $x_{n-1}\cdots x_j$.
Moreover, $j\leqslant m$ if and only if in the word $w$ we have
$j_s=m$ (see the description of possible types of words). It
follows that after multiplying by $x_{m-1}$ the sum of all
elements with $j_s=m$ vanishes.

Assume that $v$, $z$ are reduced words such that
$j_{s_v}=j_{s_z}=m$ (here $j_{s_v}, j_{s_z}$ are  defined for $v$
and $z$ as in the list of possible types (i)-(vi) listed before
Lemma~\ref{6types}) and $vx_{m-1}=zx_{m-1}$ holds in $C_n$. We use
the proof of Corollary~\ref{m=j_s} to conclude that $v=z$. Let $u$
be the reduced form of $vx_{m-1}=zx_{m-1}$.
\begin{itemize}
    \item Assume $u$ is of type $(i)$, $(iii)$, $(iv)$ or $(v)$. If $u$ has a suffix $x_{m-1}$,
    then it follows that $j_{s-1}<j_s-1$ and $vx_{m-1}$, $zx_{m-1}$ are reduced, so that $v=z$.
    Otherwise $vx_{m-1}=v$ and $zx_{m-1}=z$, so also $v=z$.
    \item Assume $u$ is of type $(vi)$. Then $j_{s-1}=j_{s}-1$ and $z=v$
    are of the form
    $(x_nq_t)^Mx_nx_1\cdots x_tx_{n-1}\ldots x_{t+2}$.
    \item Assume $u$ is of type $(ii)$.

    If $m\neq t+1$, then it follows that $\suff(u)=x_jx_nx_1$ for $j\in\{m-1, m\} $.
    If $j=m$, then $v=vx_{m-1}$ and $zx_{m-1}=z$ in $C_n$, so the assertion holds.
    If $j=m-1$, it follows that for $v=v_0x_nx_1$, we have
    $vx_{m-1}=v_0x_{m-1}x_nx_1$ in $C_n$ and the latter word is the reduced form.
    It is clear that $v=z$ also in this case.

    Otherwise $m=t+1$. Then we consider only part 3. in Corollary~\ref{m=j_s}.
   It is clear that in this case if $vx_{m-1}=zx_{m-1}$, then $v=z$.
\end{itemize}

 We have shown that for every pair of
words $v,z$ with $j_{s_v}=j_{s_z}=m$ if $vx_{m-1}=zx_{m-1}$, then
$v=z$. This implies that $\supp (\alpha)$ has no words with
$j_s=m$, which contradicts the definition of $m$. Hence, the
assertion of Proposition~\ref{anihilator} has been
proved.\\

We will also need another technical observation.

\begin{lemma} \label{x1 razy}
We have $x_{1}M_{t}\subseteq M_{t}\cup I_{t}$.
\end{lemma}
\begin{proof}
Using the involution $\chi_t $ from Corollary~\ref{wn4}, since
$\chi_t (I_{t})\subseteq I_{t}$ by Lemma~\ref{lem14}, it is enough
to show that $M_tx_t\subseteq M_{t}\cup I_{t}$, or in other words
$(x_nq_t)bx_t \in \langle x_nq_t\rangle B_t \cup I_{t}$ for every
$b\in B_{t}$. This is easily shown by applying the method used in
the proof of Case II (pushing $x_t$ to the left in the considered
word and then applying cases 4. and 5. of Lemma~\ref{multipl}).
\end{proof}

For any $K$-algebra $A$, let ${\mathcal P}(A)$ denote the prime
radical of $A$.

\begin{theorem} \label{semiprime}
For  every $t=0,1,\ldots, n-2$, the algebra $K_{0}[M_{t}]$ is
prime.
\end{theorem}

\begin{proof} In view of Remark~\ref{semiprimeM0}, $K_{0}[M_{n-2}]$ and
$K_{0}[M_{0}]$ are prime. In particular, the result holds for
$n=3$. We proceed by induction on $n$. Assume that $n>3$.
Moreover, we may assume that $1\leqslant t\leqslant n-3$.

First, we show that $K_{0}[M_{t}]$ is semiprime.  Suppose that
$\alpha K_{0}[M_{t}] \alpha =0 $ for some nonzero $\alpha \in
K_{0}[M_t]$. Then, by Theorem~\ref{tw3}, for every $u,w\in M_{t}$
we have $\supp(u\alpha w) \subseteq a\langle x_n q_t\rangle b$ for
some $a\in A_t, b\in B_t$. By Corollary~\ref{multip-rows}, if
$u\alpha w\neq 0$, then there exist $u',w'\in M_{t}$ such that
$0\neq u'u\alpha ww'\in \langle x_n q_t\rangle$, and $u'u\alpha
ww'\in {\mathcal P}(K[\langle x_n q_t\rangle] )=0$. It follows
that $u\alpha w=0$ for every $u,w\in M_t$.  Thus, either $\alpha
M_t=0$ or $M_{t}\alpha w=0$ for some $w\in M_t$ such that $\alpha
w\neq 0$.
 This means that $\alpha \circ P_t=0$ or $P_{t}\circ
\alpha w =0$ ($\circ $ stands for the ordinary matrix
multiplication, where $K_{0}[M_t]$ is interpreted as a subset of
the matrix algebra $M_{|A_t|}(K[\langle x_n q_t\rangle])$. Since
$P_t$ is a symmetric matrix by Remark~\ref{symmetric}, we may
assume that $\alpha \circ P_t=0$ for some nonzero $\alpha \in
K_{0}[M_{t}]$. Then $\alpha $ can be chosen so that $\supp
(\alpha)\subseteq a \langle x_n q_t\rangle B_{t}$ for some $a\in
A_{t}$. Hence, Corollary~\ref{multip-rows} allows us to assume
that $\supp (\alpha)\subseteq \langle x_n q_t\rangle B_{t}$.
Finally, we may assume that $|\supp (\alpha)|$ is minimal
possible.

We claim that $\alpha x_{1}=0$ in $K_{0}[M_t]$. By Lemma~\ref{x1
razy}, $\alpha x_{1}M_{t}=0$ in $K_{0}[M_t]$. From the proof of
Lemma~\ref{6types} we know that $ \alpha_{(iv)}x_1 =0$ in
$K_{0}[M_t]$ and $vx_1\in \langle x_n q_t\rangle B_{t}$ for every
$v\in \supp (\alpha)\setminus \supp(\alpha_{(iv)})$. So, $\alpha
x_1$ inherits the hypotheses on $\alpha$. Therefore, the minimal
choice of $\alpha$ allows us to assume that $\alpha_{(iv)}=0$.
Moreover, $\alpha x_{1}\in K[\langle x_2,\ldots, x_{n-1}, x_nx_1
\rangle ]$. But, from Lemma~\ref{induction} we know that the
latter is isomorphic to $K[C_{n-1}]$.  Moreover, under this
identification, $\supp (\alpha x_1) $  is contained in a single
row of the matrix structure $M_{t-1}^{(n-1)}$  defined for the
monoid $C_{n-1}$ as in Definition~\ref{matrix}. It is easy to see
that $\alpha x_1 M_{t-1}^{(n-1)} =0 $ in $K_{0}[M_{t-1}^{(n-1)}]$.
The inductive hypothesis implies that $\alpha x_1=0$. This proves
the claim.

From  Lemma~\ref{sigma} it follows that replacing $\alpha$ by
$(x_nq_t)^k\alpha$, for some $k\geqslant 1$, if necessary, we may
assume that $\sigma (\alpha )\in M_t$ and hence we get that
$\sigma (\alpha )$ lies in a single row of the matrix structure
$K_{0}[M_t]$. In other words, there exists $a\in A_t$ such that
$\supp (\alpha)\subseteq a\langle x_nq_t\rangle B_t$.
 Then, by
Corollary~\ref{multip-rows}, there exists $z\in M_t$ such that
$\supp (z\sigma (\alpha )) \subseteq \langle x_n q_t\rangle
B_{t}$. The proof of Lemma~\ref{6types} implies that for every
$w\in \supp (\alpha)$ either $wx_{1}\in a\langle x_n q_t\rangle
B_{t}$ or $wx_1 = 0$ in $K_{0}[M_t]$. Therefore, by the previous
paragraph, $z\sigma (\alpha) x_{1}=0$. Hence,
Lemma~\ref{multip-rows} implies that also $\sigma (\alpha x_{n})=
\sigma (\alpha) x_{1}=0$. Consequently, $\alpha x_{n}=0$.

Repeating this argument, we get that $\alpha x_{i}=0$ in
$K_{0}[M_{t}]$ for every $i=1,\ldots, n$. From
Proposition~\ref{anihilator} it now follows that $\alpha =0$, a
contradiction. Thus, we have proved that $K_{0}[M_t]$ is
semiprime. This implies that the sandwich matrix $P_{t}$ is not a
zero divisor (in the corresponding matrix ring $M_{n_t}(K[\langle
x_nq_t\rangle ])$, where $n_{t}=|A_{t}|$). Since $K[\langle
x_nq_t\rangle ]$ is a domain, it is known that $K_{0}[M_t]$ must
be prime, see \cite{semalg}, Chapter~5.
\end{proof}

 It follows that every matrix $P_{t}, t=0,\ldots, n-2$, has a
nonzero determinant, which seems to be inaccessible by a direct
proof. We are now ready for the main results of this section.

\begin{theorem} \label{noether-cycle}
For every $n\geqslant 3$ the algebra $K[C_n]$ is right and left
Noetherian.
\end{theorem}
\begin{proof}
From  \cite{stafford} we know that the algebra $K[C_{n}]/{\mathcal
P}(K[C_{n}])$ is right and left Noetherian because (by
\cite{mecel_okninski1}) it is a semiprime algebra of
Gelfand--Kirillov dimension $1$.

Suppose that $0\neq \beta \in {\mathcal P}(K[C_{n}])\cap K[\tilde
M]$, where $\tilde M=\bigcup_{i=0}^{n-2}\tilde M_{i}$, as defined
in Section~\ref{ideal}. Let $i$ be the minimal integer such that
$\supp (\beta)\cap \tilde M_{i}\neq \emptyset$. Passing to
$K[C_{n}]/I_{i}$, we get a nonzero element $\overline{\beta} \in
{\mathcal P}(K_{0}[C_n/I_i])\cap K_{0}[M_{i}] \subseteq {\mathcal
P}(K_{0}[M_{i}])$ (see Corollary~\ref{struktura}). This
contradicts Theorem~\ref{semiprime}. Hence, ${\mathcal
P}(K[C_{n}])\cap K[\tilde M]=0$. Since $C_n\setminus \tilde M$ is
finite by Proposition~\ref{stw1}, it follows that ${\mathcal
P}(K[C_n])$ is finite dimensional. Therefore, $K[C_{n}]$ also is
right and left Noetherian.
\end{proof}

\begin{theorem}\label{Noetherian} Let $\Theta$ be a finite oriented graph. Then the
following conditions are equivalent
\begin{enumerate}
\item[1)] $K[\HK_{\Theta}]$ is right Noetherian,
\item[2)] $K[\HK_{\Theta}]$ is left Noetherian,
\item[3)] each of the connected components of $\Theta$ is either an oriented cycle or an acyclic graph.
\end{enumerate}
\end{theorem}

\begin{proof}
Assume that condition 3) is satisfied. From \cite{mecel_okninski2}
we know that $\HK_{\Theta}$ is a PI--algebra. In order to prove
conditions 1) and  2) we proceed by induction on the number $k$ of
connected components of $\Theta$. If $k=1$ then the assertion
follows from Theorem~\ref{noether-cycle} and from the fact that
$\HK_{\Theta}$ is finite if $\Theta$ is an acyclic graph. Assume
that $k>1$. Let $\Theta_1$ be a connected component of $\Theta $
and let $ \Theta_2 = \Theta \setminus \Theta_1$. Clearly,
$\HK_{\Theta }$ is a direct product of $\HK_{\Theta_1}$ and
$\HK_{\Theta_2}$, so that $K[\HK_{\Theta }]\cong
K[\HK_{\Theta_1}]\otimes K[\HK_{\Theta_2}]$. By the induction
hypothesis, $\HK_{\Theta_i}$ is (right and left) Noetherian and it
is a PI--algebra, for $i=1,2$. Then $K[\HK_{\Theta}]$ is a
Noetherian algebra by \cite{artin}, Proposition 4.4 (which says
that every finitely generated right Noetherian PI--algebra is a
universally right Noetherian algebra).

Assume that 3) is not satisfied. Then $\Theta$ contains a subgraph
$\Theta'$ that is of the form described in Lemma~\ref{nonnoether}
or the graph $\Theta ''$ obtained from $\Theta'$ by inverting all
arrows. It is easy to see that in this case $K[\HK_{\Theta'}]$,
respectively $K[\HK_{\Theta ''}]$, is a homomorphic image of
$K[\HK_{\Theta}]$, as noticed in \cite{type A}. Moreover,
$\Theta'$ and $\Theta ''$ are antiisomorphic. Therefore,
Lemma~\ref{nonnoether} implies that $K[\HK_{\Theta}]$ is neither
right nor left Noetherian. The result follows.
\end{proof}

From the proof it actually follows that the conditions in
Theorem~\ref{Noetherian} are satisfied if and only if the monoid
$\HK_{\Theta}$ has acc on right (left) ideals.

Since $K[C_n]$ is a PI--algebra \cite{mecel_okninski1}, we derive
the following direct consequence from the result of Anan'in
\cite{ananin}. It is of interest in view of the results on
faithful representations of various special classes of
Hecke--Kiselman monoids, obtained in \cite{type
A},\cite{maz},\cite{OnK}.

\begin{corollary}
$K[C_n]$ embeds into the matrix algebra $M_{r}(L)$ over a field
$L$, for some $r\geqslant 1$.
\end{corollary}

We conclude with some open questions.

\begin{question} Is $K[C_n]$ semiprime for every $n\geqslant 3$?
\end{question}

If this is the case, the main result of \cite{stafford} allows to
strengthen the assertion of Theorem~\ref{noether-cycle}: in this
case $K[C_n]$ is a finitely generated module over its Noetherian
center; and hence also of Theorem~\ref{Noetherian}. Notice that
this is the case if $n=3$, by Theorem~\ref{semiprimeC3}. Moreover,
the proof of Theorem~\ref{noether-cycle} shows that ${\mathcal
P}(K[C_n])$ is finite dimensional.

Because of the main results of Section~\ref{ideal}, algebras
$K[C_n]$ share the flavor of affine cellular algebras, introduced
in \cite{koenig-xi}, also see \cite{lomp}. This motivates our
second question.
\begin{question} Does $K[C_n]$ admit a structure of a cellular
algebra?
\end{question}

\end{document}